\documentclass[11pt]{amsart}
\usepackage{amssymb,amsfonts,amsthm,amscd,amsmath,amsgen,upref}

\usepackage[margin=1in]{geometry}

\theoremstyle{plain}
\newtheorem{cor}{Corollary}

\newtheorem{prop}[cor]{Proposition}
\newtheorem{con}[cor]{Control}
\newtheorem{thm}[cor]{Theorem}
\theoremstyle{definition}

\numberwithin{cor}{section}
\numberwithin{equation}{section}

\DeclareMathOperator{\tr}{tr}

\DeclareMathOperator{\C}{C}
\DeclareMathOperator{\USC}{USC}
\DeclareMathOperator{\LSC}{LSC}
\DeclareMathOperator{\BUC}{BUC}
\DeclareMathOperator{\Lip}{Lip}

\DeclareMathOperator{\Supp}{Supp}

\renewcommand{\d}{d} 

\newcommand{\abs}[1]{\lvert#1\rvert}
\newcommand{\norm}[1]{\lVert#1\rVert}

\def\XXint#1#2#3{{\setbox0=\hbox{$#1{#2#3}{\int}$ }
\vcenter{\hbox{$#2#3$ }}\kern-.6\wd0}}

\title{A Liouville Property for Isotropic Diffusions in Random Environment}

\author{Benjamin J. Fehrman}

\date{June 5, 2014}



\address{Department of Mathematics, The University of Chicago, 5734 S. University Avenue, Chicago IL, 60637.}

\email{bfehrman@math.uchicago.edu}

\begin{document}

\begin{abstract}
We obtain a Liouville property for stationary diffusions in random environment which are small, isotropic perturbations of Brownian motion in spacial dimension greater than two.  Precisely, we prove that, on a subset of full probability, the constant functions are the only strictly sub-linear maps which are invariant with respect to the evolution of the diffusion.  And, we prove that the constant functions are the only bounded, ancient maps which are invariant under the evolution.  These results depend upon the previous work of Fehrman \cite{F1} and Sznitman and Zeitouni \cite{SZ} and, in the first case, our methods are motivated by the work, in the discrete setting, of Benjamini, Duminil-Copin, Kozma and Yadin \cite{BDKY}.
\end{abstract}

\maketitle

\section{Introduction}

In this paper, we establish a Liouville property for stationary diffusions in random environment which are small, isotropic perturbations of Brownian motion in dimensions greater than two.  Precisely, there exists a probability space $(\Omega,\mathcal{F},\mathbb{P})$ indexing the collection of environments described, for each $x\in\mathbb{R}^d$ and $\omega\in\Omega$, by coefficients $$A(x,\omega)\in\mathcal{S}(d)\;\;\textrm{and}\;\;b(x,\omega)\in \mathbb{R}^d,$$ where we assume, in particular, that the processes $A(x,\omega)$ and $b(x,\omega)$ are stationary and satisfy a finite range dependence and restricted isotropy condition.  That is, whenever subsets $A,B\subset\mathbb{R}^d$ are sufficiently distant, the sigma algebras $$\sigma\left( A(x,\omega), b(x,\omega)\;|\;x\in A\right)\;\;\textrm{and}\;\;\sigma\left(A(x,\omega), b(x,\omega)\;|\;x\in B\right)\;\;\textrm{are independent.}$$  And, whenever $r:\mathbb{R}^d\rightarrow\mathbb{R}^d$ is an orthogonal transformation preserving the coordinate axis, for each $x\in\mathbb{R}^d$, the random variables $$\left(A(rx,\omega),b(rx,\omega)\right)\;\;\textrm{and}\;\;\left(rA(x,\omega)r^t,rb(x,\omega)\right)\;\;\textrm{have the same law.}$$  We furthermore assume that the process is a small perturbation of Brownian motion.  For $\eta>0$ to be chosen sufficiently small, $$\abs{A(x,\omega)-I}\leq \eta\;\;\textrm{and}\;\;\abs{b(x,\omega)}\leq \eta\;\;\textrm{on}\;\;\mathbb{R}^d\times\Omega.$$

The precise statement of these and additional assumptions may be found in Section 2.  Observe there that the assumptions are identical to those first considered by Sznitman and Zeitouni \cite{SZ}, which correspond to the continuous analogue of those considered in the discrete setting by Bricmont and Kupiainen \cite{BK}.

Before stating our main result, we remark that there exists, see Friedman \cite{Fr}, for each $\omega\in\Omega$, a Green's function $p_{t,\omega}(x,y):\mathbb{R}^d\times\mathbb{R}^d\times(0,\infty)\rightarrow\mathbb{R},$ such that, for continuous initial data growing, for instance, at most quadratically, the solution $w:\mathbb{R}^d\times[0,\infty)\rightarrow\mathbb{R}$ satisfying \begin{equation}\label{i_par}\left\{\begin{array}{ll}w_t-\frac{1}{2}\tr(A(x,\omega)D^2w)+b(x,\omega)\cdot Dw=0 & \textrm{on}\;\;\mathbb{R}^d\times(0,\infty), \\ w=f & \textrm{on}\;\;\mathbb{R}^d\times\left\{0\right\},\end{array}\right.\end{equation} admits the representation \begin{equation}\label{i_par_rep}w(x,t)=\int_{\mathbb{R}^d}p_{t,\omega}(x,y)f(y)\;dy\;\;\textrm{on}\;\;\mathbb{R}^d\times[0,\infty).\end{equation}  Similarly, see Stroock and Varadhan \cite{SV}, for each $\omega\in\Omega$ and $x\in\mathbb{R}^d$, the martingale problem corresponding to the generator $$\frac{1}{2}\sum_{i,j=1}^da_{ij}(y,\omega)\frac{\partial^2}{\partial y_i\partial y_j}-\sum_{i=1}^db(y,\omega)\frac{\partial}{\partial y_i}$$ is well-posed.  We denote by $P_{x,\omega}$ the corresponding probability measures on the space of continuous paths $\C([0,\infty);\mathbb{R}^d)$ and observe that the solution of (\ref{i_par}) admits the representation \begin{equation}\label{i_mart} w(x,t)=P_{x,\omega}\left(f(X_t)\right)\;\;\textrm{on}\;\;\mathbb{R}^d\times[0,\infty).\end{equation}  We may now summarize the result.

The purpose of this paper is to prove that, on a subset of full probability, the constant functions are the only strictly sub-linear solutions $w:\mathbb{R}^d\rightarrow\mathbb{R}$ to the time-independent problem \begin{equation}\label{i_eq} -\frac{1}{2}\tr(A(x,\omega)D^2 w)+b(x,\omega)\cdot Dw=0\;\;\textrm{on}\;\;\mathbb{R}^d,\end{equation} where, in view of (\ref{i_par_rep}) and (\ref{i_mart}), we will say that a strictly sub-linear $w:\mathbb{R}^d\rightarrow\mathbb{R}$ satisfies (\ref{i_eq}) if, for each $x\in\mathbb{R}^d$ and $t\geq 0$, \begin{equation}\label{i_sol}w(x)=\int_{\mathbb{R}^d}p_{t,\omega}(x,y)w(y)\;dy=P_{x,\omega}\left(w(X_t)\right).\end{equation}  We recall that a function $w:\mathbb{R}^d\rightarrow\mathbb{R}$ is strictly sub-linear if $$\lim_{\abs{y}\rightarrow\infty}\frac{w(y)}{\abs{y}}=0,$$ and now state our main theorem.

\begin{thm}\label{i_main}  Assume (\ref{steady}) and (\ref{constants}).  There exists a subset of full probability on which the constant functions are the only strictly sub-linear $w:\mathbb{R}^d\rightarrow\mathbb{R}$ satisfying (\ref{i_sol}).\end{thm}

Furthermore, we obtain the identical statement for bounded, ancient solutions to the time-independent problem.  In this case, we are concerned with bounded solutions $w:\mathbb{R}^d\times( -\infty,\infty)\rightarrow\mathbb{R}$ satisfying \begin{equation}\label{i_ancient} -\frac{1}{2}\tr(A(x,\omega)D^2w)+b(x,\omega)\cdot Dw=0\;\;\textrm{on}\;\;\mathbb{R}^d\times(0,\infty),\end{equation} where we will say that a bounded $w:\mathbb{R}^d\times(-\infty,\infty)\rightarrow\mathbb{R}$ satisfies (\ref{i_ancient}) if, for each $x\in\mathbb{R}^d$, $s\geq 0$ and $t\in (-\infty,\infty)$, \begin{equation}\label{i_ancient_sol} w(x,t+s)=\int_{\mathbb{R}^d}p_{s,\omega}(x,y)w(y,t)\;dy=P_{x,\omega}\left(w(X_s,t)\right).\end{equation}  The theorem follows.

\begin{thm}\label{i_ancient}  Assume (\ref{steady}) and (\ref{constants}).  There exists a subset of full probability on which the constant functions are the only bounded $w:\mathbb{R}^d\times(-\infty,\infty)\rightarrow\mathbb{R}$ satisfying (\ref{i_ancient_sol}).\end{thm}

The proof of both results rely upon the previous work of Fehrman \cite{F1} and \cite{SZ}.  And, the proof of Theorem \ref{i_main} is motivated by the methods of Benjamini, Duminil-Copin, Kozma and Yadin \cite{BDKY}, where it was shown in the discrete setting that, in the presence of an invariant measure, a Liouville property may be obtained for the environment whenever we have an effective control, in expectation, of the ensemble of limiting diffusivities $$\limsup_{t\rightarrow\infty} \frac{1}{td}P_{0,\omega}\left(\abs{X_t}^2\right),$$ and, a control from above of the physical entropies determined by the Green's functions, as defined, for each $x\in\mathbb{R}^d$, $t> 0$ and $\omega\in\Omega$, as $$H_{t,\omega}(x)=\int_{\mathbb{R}^d}-p_{t,\omega}(x,y)\log(p_{t,\omega}(x,y))\;dy.$$

To this end, using the results of \cite{SZ}, we prove in Section 3 that, on a subset of full probability, for $\overline{\alpha}>0$ identified in Theorem \ref{effectivediffusivity}, $$\lim_{t\rightarrow\infty}\frac{1}{td}P_{0,\omega}\left(\abs{X_t}^2\right)=\overline{\alpha}.$$  And, in Section 4, we prove under general assumptions that the physical entropy grows at most logarithmically.  Precisely, for each $\omega\in\Omega$, for $C>0$ independent of $t\geq 1$, we have $$H_{t,\omega}(0)=\int_{\mathbb{R}^d}-p_{t,\omega}(0,y)\log(p_{t,\omega}(0,y))\;dy\leq C\left(\log(t)+1\right).$$  Finally, the existence of an invariant measure for the environment was shown in \cite{F1}, which we recall together with the results of \cite{SZ} in Section 2.  The proof of Theorem \ref{i_main} is presented in Section 5 and, in Section 6, we use the methods of \cite{F1} to prove Theorem \ref{i_ancient}.

\subsection*{Acknowledgments}

I would like to thank Professor Panagiotis Souganidis for suggesting this problem,  and I would like to thank Professors Panagiotis Souganidis, Ofer Zeitouni and Luis Silvestre for many useful conversations.

\section{Preliminaries} 

\subsection{Notation}

Elements of $\mathbb{R}^d$ and $[0,\infty)$ are denoted by $x$ and $y$ and $t$ respectively and $(x,y)$ denotes the standard inner product on $\mathbb{R}^d$.  We write $Dv$ and $v_t$ for the derivative of the scalar function $v$ with respect to $x\in\mathbb{R}^d$ and $t\in[0,\infty)$, while $D^2v$ stands for the Hessian of $v$.  The spaces of $k\times l$ and $k\times k$ symmetric matrices with real entries are respectively written $\mathcal{M}^{k\times l}$ and $\mathcal{S}(k)$.  If $M\in\mathcal{M}^{k\times l}$, then $M^t$ is its transpose and $\abs{M}$ is its norm $\abs{M}=\tr(MM^t)^{1/2}.$  If $M$ is a square matrix, we write $\tr(M)$ for the trace of $M$.  The Euclidean distance between subsets $A,B\subset\mathbb{R}^d$ is $$d(A,B)=\inf\left\{\;\abs{a-b}\;|\;a\in A, b\in B\;\right\}$$  and, for an index $\mathcal{A}$ and a family of measurable functions $\left\{f_\alpha:\mathbb{R}^d\times\Omega\rightarrow\mathbb{R}^{n_\alpha}\right\}_{\alpha\in\mathcal{A}}$, we write $$\sigma(f_\alpha(x,\omega)\;|\;x\in A, \alpha\in\mathcal{A})$$ for the sigma algebra generated by the random variables $f_\alpha(x,\omega)$ for $x\in A$ and $\alpha\in\mathcal{A}$.  For $U\subset\mathbb{R}^d$, $\USC(U;\mathbb{R}^d)$, $\LSC(U;\mathbb{R}^d)$, $\BUC(U;\mathbb{R}^d)$, $\C(U;\mathbb{R}^d)$, $\Lip(U;\mathbb{R}^d)$, $\C^{0,\beta}(U;\mathbb{R}^d)$ and $\C^k(U;\mathbb{R}^d)$ are the spaces of upper-semicontinuous, lower-semicontinuous, bounded continuous, continuous, Lipschitz continuous, $\beta$-H\"{o}lder continuous and $k$-continuously differentiable functions on $U$ with values in $\mathbb{R}^d$.  For $f:\mathbb{R}^d\rightarrow\mathbb{R}$, we write $\Supp(f)$ for the support of $f$.  Furthermore, $B_R$ and $B_R(x)$ are respectively the open balls of radius $R$ centered at zero and $x\in\mathbb{R}^d$.  For a real number $r\in\mathbb{R}$ we write $\left[r\right]$ for the largest integer less than or equal to $r$.  Finally, throughout the paper we write $C$ for constants that may change from line to line but are independent of $\omega\in\Omega$ unless otherwise indicated.

\subsection{The Random Environment}

There exists an underlying probability space $(\Omega,\mathcal{F},\mathbb{P})$ indexing the individual realizations of the random environment.   Since the environment is described, for each $x\in\mathbb{R}^d$ and $\omega\in\Omega$, by the diffusion matrix $A(x,\omega)$ and drift $b(x,\omega)$, we may take \begin{equation}\label{sigmaalgebra} \mathcal{F}=\sigma\left(A(x,\omega),b(x,\omega)\;|\;x\in\mathbb{R}^d\right).\end{equation}  Furthermore, we assume this space is equipped with an \begin{equation}\label{transgroup} \textrm{ergodic group of measure-preserving transformations}\; \left\{\tau_x:\Omega\rightarrow\Omega\right\}_{x\in\mathbb{R}^d},\end{equation} such that the coefficients $A:\mathbb{R}^d\times\Omega\rightarrow\mathcal{S}(d)$ and $b:\mathbb{R}^d\times\Omega\rightarrow\mathbb{R}^d$ are bi-measurable stationary functions satisfying, for each $x,y\in\mathbb{R}^d$ and $\omega\in\Omega$, \begin{equation}\label{stationary} A(x+y,\omega)=A(x,\tau_y\omega)\;\;\textrm{and}\;\;b(x+y,\omega)=b(x,\tau_y\omega).\end{equation}

We assume that the diffusion matrix and drift are bounded and Lipschitz uniformly for $\omega\in\Omega$.  There exists $C>0$ such that, for all $y\in\mathbb{R}^d$ and $\omega\in\Omega$,  \begin{equation}\label{bounded} \abs{b(y,\omega)}\leq C\;\;\;\textrm{and}\;\;\;\abs{A(y,\omega)}\leq C \end{equation} and, for all $x,y\in\mathbb{R}^d$ and $\omega\in\Omega$, \begin{equation}\label{Lipschitz} \abs{b(x,\omega)-b(y,\omega)}\leq C\abs{x-y}\;\;\;\textrm{and}\;\;\;\abs{A(x,\omega)-A(y,\omega)}\leq C\abs{x-y}.\end{equation}  In addition, we assume that the diffusion matrix is uniformly elliptic uniformly in $\Omega$.  There exists $\nu>1$ such that, for all $y\in\mathbb{R}^d$ and $\omega\in\Omega$, \begin{equation}\label{elliptic} \frac{1}{\nu} I\leq A(y,\omega)\leq \nu I.\end{equation}

The coefficients satisfy a finite range dependence.  There exists $R>0$ such that, whenever $A,B\subset\mathbb{R}^d$ satisfy $d(A,B)\geq R$, the sigma algebras \begin{equation}\label{finitedep} \sigma(A(x,\omega), b(x,\omega)\;|\;x\in A)\;\;\;\textrm{and}\;\;\; \sigma(A(x,\omega), b(x,\omega)\;|\;x\in B)\;\;\;\textrm{are independent.}\end{equation}  The diffusion matrix and drift satisfy a restricted isotropy condition.  For every orthogonal transformation $r:\mathbb{R}^d\rightarrow\mathbb{R}^d$ which preserves the coordinate axes, for every $x\in\mathbb{R}^d$, \begin{equation}\label{isotropy} (A(rx,\omega),b(rx,\omega))\;\;\;\textrm{and}\;\;\;(rA(x,\omega)r^t,rb(x,\omega))\;\;\;\textrm{have the same law.}\end{equation}  And, finally, the diffusion matrix and drift are a small perturbation of the Laplacian.  There exists $\eta_0>0$, to later be chosen small, such that, for all $y\in\mathbb{R}^d$ and $\omega\in\Omega$, \begin{equation}\label{perturbation} \abs{b(y,\omega)}\leq\eta_0\;\;\textrm{and}\;\;\abs{A(y,\omega)-I}\leq \eta_0.\end{equation}

To avoid cumbersome statements in what follows, we introduce a steady assumption.  \begin{equation}\label{steady}\textrm{Assume}\;(\ref{sigmaalgebra}),(\ref{transgroup}), (\ref{stationary}), (\ref{bounded}), (\ref{Lipschitz}), (\ref{elliptic}), (\ref{finitedep}), (\ref{isotropy})\;\textrm{and}\;(\ref{perturbation}).\end{equation}

The collection of assumptions (\ref{transgroup}), (\ref{stationary}), (\ref{bounded}), (\ref{Lipschitz}) and (\ref{elliptic}) guarantee the well-posedness of the martingale problem set on $\mathbb{R}^d$ for each $\omega\in\Omega$ and $x\in\mathbb{R}^d$, see \cite{SV}, associated to to the generator $$\frac{1}{2}\sum_{i,j=1}^da_{ij}(y,\omega)\frac{\partial^2}{\partial y_i\partial y_j}-\sum_{i=1}^db_i(y,\omega)\frac{\partial}{\partial y_i}.$$  We write $P_{x,\omega}$ and $E_{x,\omega}$ for the corresponding probability measure and expectation on the space of continuous paths $\C([0,\infty);\mathbb{R}^d)$ and remark that, almost surely with respect to $P_{x,\omega}$, paths $X_t\in\C([0,\infty);\mathbb{R}^d)$ satisfy the stochastic differential equation \begin{equation}\label{sde}\left\{\begin{array}{l} dX_t=-b(X_t,\omega)dt+\sigma(X_t,\omega)dB_t, \\ X_0=x,\end{array}\right.\end{equation} for $A(y,\omega)=\sigma(y,\omega)\sigma(y,\omega)^t$, and for $B_t$ a standard Brownian motion under $P_{x,\omega}$ with respect to the canonical right-continuous filtration on $\C([0,\infty);\mathbb{R}^d)$.

We write $\mathbb{P}_x=\mathbb{P}\ltimes P_{x,\omega}$ and $\mathbb{E}_x=\mathbb{E}\ltimes E_{x,\omega}$ for the corresponding semi-direct product measure and expectation on $\Omega\times\C([0,\infty);\mathbb{R}^d)$.  The annealed law $\mathbb{P}_x$ inherits the translation invariance and restricted rotational invariance implied by (\ref{stationary}) and (\ref{isotropy}).  In particular, for all $x,y\in\mathbb{R}^d$, \begin{equation}\label{annealed} \mathbb{E}_{x+y}(X_t)=\mathbb{E}_y(x+X_t)=x+\mathbb{E}_y(X_t),\end{equation} and, for all orthogonal transformations $r$ preserving the coordinate axis and for every $x\in\mathbb{R}^d$, \begin{equation}\label{annealed1} \mathbb{E}_{x}(rX_t)=\mathbb{E}_{rx}(X_t).\end{equation}   This stands in contrast to the quenched laws $P_{x,\omega}$, for which no invariance properties can be expected to hold, in general.

\subsection{A Review of \cite{SZ}}  In this section, we review the aspects of \cite{SZ} most relevant to our arguments.  Observe that this summary is by no means complete, as considerably more was achieved in their paper than we mention here.

We are interested in the long term behavior of the equation, for a fixed, H\"older continuous function $f:\mathbb{R}^d\rightarrow\mathbb{R}$, \begin{equation}\label{review_eq}\left\{\begin{array}{ll} u_t-\frac{1}{2}\tr(A(x,\omega)D^2u)+b(x,\omega)\cdot Du=0 & \textrm{on}\;\;\mathbb{R}^d\times(0,\infty), \\ u=f & \textrm{on}\;\;\mathbb{R}^d\times\left\{0\right\}.\end{array}\right.\end{equation} This is essentially achieved by comparing the solutions of (\ref{review_eq}) to the solution of the deterministic problem, for $\overline{\alpha}>0$ identified in Theorem \ref{effectivediffusivity}, \begin{equation}\label{review_hom}\left\{\begin{array}{ll} \overline{u}_t-\frac{\overline{\alpha}}{2}\Delta \overline{u}=0 & \textrm{on}\;\;\mathbb{R}^d\times(0,\infty), \\ \overline{u}=f & \textrm{on}\;\;\mathbb{R}^d\times\left\{0\right\},\end{array}\right.\end{equation} along an increasing sequence of length and time scales.

The constant $\overline{\alpha}$ determining (\ref{review_hom}) is identified in \cite{SZ} through a process we describe after introducing some notation.  Fix the dimension \begin{equation}\label{dimension} d\geq 3, \end{equation} and fix a H\"older exponent \begin{equation}\label{Holderexponent} \beta\in\left(0,\frac{1}{2}\right]\;\;\textrm{and a constant}\;\;a\in \left(0,\frac{\beta}{1000d}\right]. \end{equation}

Let $L_0$ be a large integer multiple of five.  For each $n\geq 0$, inductively define \begin{equation}\label{L} \ell_n=5\left[\frac{L_n^a}{5}\right]\;\;\textrm{and}\;\;L_{n+1}=\ell_n L_n, \end{equation} so that, for $L_0$ sufficiently large, we have $\frac{1}{2}L_n^{1+a}\leq L_{n+1}\leq 2L_n^{1+a}$.  For each $n\geq 0$, for $c_0>0$, let \begin{equation}\label{kappa} \kappa_n=\exp(c_0(\log\log(L_n))^2)\;\;\textrm{and}\;\;\tilde{\kappa}_n=\exp(2c_0(\log\log(L_n))^2),\end{equation} where we remark that, as $n$ tends to infinity, $\kappa_n$ is eventually dominated by every positive power of $L_n$.  Furthermore, define, for each $n\geq 0$, \begin{equation}\label{D} D_n=L_n\kappa_n\;\;\textrm{and}\;\;\tilde{D}_n=L_n\tilde{\kappa}_n.\end{equation}  We choose $L_0$ sufficiently large such that, for each $n\geq 0$, \begin{equation}\label{D_1} L_n<D_n< \tilde{D}_n< L_{n+1},\;\;4\tilde{\kappa}_n<\tilde{\kappa}_{n+1}\;\;\textrm{and}\;\; 3\tilde{D}_{n+1} < L_{n+1}^2.\end{equation}

The following constants enter into the probabilistic statements below.  Fix $m_0\geq 2$ satisfying \begin{equation}\label{m0} (1+a)^{m_0-2}\leq 100<(1+a)^{m_0-1}, \end{equation}  and $\delta>0$ and $M_0>0$ satisfying \begin{equation}\label{delta} \delta=\frac{5}{32}\beta\;\;\textrm{and}\;\;M_0\geq100d(1+a)^{m_0+2}.\end{equation}  In the arguments to follow, we will use the fact that $\delta$ and $M_0$ are sufficiently larger than $a$.

We now describe the identification of $\overline{\alpha}$.  Recall, for each $x\in\mathbb{R}^d$ and $\omega\in\Omega$, the quenched law $P_{x,\omega}$ on $\C([0,\infty);\mathbb{R}^d)$ and, for each $x\in\mathbb{R}^d$, the annealed law $\mathbb{P}_x$ on $\Omega\times\C([0,\infty);\mathbb{R}^d)$.  The constant $\overline{\alpha}$ is effectively identified as the limit of the effective diffusivities, in average, of the ensemble of equations (\ref{review_eq}) along the sequence of time steps $L_n^2$.  However, so as to apply the finite range dependence, see (\ref{finitedep}), the stopping time \begin{equation}\label{stopping} T_n=\inf\left\{\;s\geq0\;|\;\abs{X_s-X_0}\geq\tilde{D}_n\;\right\}\end{equation} is introduced, for each $n\geq 0$, and the approximate effective diffusivity of ensemble (\ref{review_eq}) is defined as \begin{equation}\label{alphan} \alpha_n=\frac{1}{dL_n^2}\mathbb{E}_0[\abs{X_{T_n\wedge L_n^2}}^2].\end{equation}  The following theorem describes the control and convergence of the $\alpha_n$ to $\overline{\alpha}$.

\begin{thm}\label{effectivediffusivity} Assume (\ref{steady}).  There exists $L_0$ and $c_0$ sufficiently large and $\eta_0>0$ sufficiently small such that, for all $n\geq 0$, $$\frac{1}{2\nu}\leq \alpha_n\leq 2\nu\;\;\textrm{and}\;\;\abs{\alpha_{n+1}-\alpha_n}\leq L_n^{-(1+\frac{9}{10})\delta},$$  which implies the existence of $\overline{\alpha}>0$ satisfying $$\frac{1}{2\nu}\leq \overline{\alpha}\leq 2\nu\;\;\textrm{and}\;\;\lim_{n\rightarrow\infty}\alpha_n=\overline{\alpha}.$$\end{thm}

We now describe the comparison between solutions of (\ref{review_eq}) and (\ref{review_hom}).  First, we compare solutions of (\ref{review_eq}), for each $n\geq 0$, at time $L_n^2$, with respect to a H\"{o}lder norm at scale $L_n$, to solutions of the deterministic problem \begin{equation}\label{review_approximate}\left\{\begin{array}{ll} u_{n,t}-\frac{\alpha_n}{2}\Delta u_n=0 & \textrm{on}\;\;\mathbb{R}^d\times(0,\infty), \\ u_{n,t}=f & \textrm{on}\;\;\mathbb{R}^d\times\left\{0\right\}.\end{array}\right. \end{equation}  To do so, we introduce, for each $n\geq 0$, the rescaled H\"older norm \begin{equation}\label{levelholder} \abs{u_0}_n=\sup_{x\in\mathbb{R}^d}\abs{u_0(x)}+L_n^\beta\sup_{x\neq y}\frac{\abs{u_0(x)-u_0(y)}}{\abs{x-y}^\beta}.\end{equation}

We will obtain a localized control of the difference between solutions of (\ref{review_eq}) and (\ref{review_approximate}) at time $L_n^2$.  This localization is obtained via a cutoff function.  For each $v>0$, let\begin{equation}\label{cutoff} \chi(y)=1\wedge(2-\abs{y})_+\;\;\textrm{and}\;\;\chi_{v}(y)=\chi\left(\frac{y}{v}\right), \end{equation} and define, for each $x\in\mathbb{R}^d$ and $n\geq 0$, \begin{equation}\label{cutoff1}  \chi_{n,x}(y)=\chi_{30\sqrt{d}L_n}(y-x).\end{equation}  The following result then describes the desired comparison between solutions of (\ref{review_eq}) and (\ref{review_approximate}), at time $L_n^2$, for H\"older continuous initial data.

\begin{con}\label{Holder}  Fix $x\in\mathbb{R}^d$, $\omega\in\Omega$ and $n\geq 0$.  Let $u$ and $u_n$ respectively denote the solutions of (\ref{review_eq}) and (\ref{review_approximate}) corresponding to initial data $f\in\C^{0,\beta}(\mathbb{R}^d)$.  We have $$\abs{\chi_{n,x}(y)\left(u(y,L_n^2)-u_n(y,L_n^2)\right)}_n\leq L_n^{-\delta}\abs{f}_n.$$\end{con}

Notice that this control depends upon $x\in\mathbb{R}^d$, $\omega\in\Omega$ and $n\geq 0$.  It is not true, in general, that this type of contraction is available for all such triples $(x,\omega,n)$.  However, as described below, it is shown in \cite{SZ} that such controls are available for large $n$, with high probability, on a large portion of space.

The final control we will use concerns tail-estimates for the diffusion process.  We wish to control, under $P_{x,\omega}$, for $X_t\in\C([0,\infty);\mathbb{R}^d)$, the probability that \begin{equation}\label{star} X_t^*=\max_{0\leq s\leq t}\abs{X_s-X_0}\end{equation} is large with respect to the time elapsed.  The desired control contained in the following proposition is similar to the standard exponential estimates for Brownian motion at large length scales.

\begin{con}\label{localization}  Fix $x\in\mathbb{R}^d$, $\omega\in\Omega$ and $n\geq 0$.  For each $v\geq D_n$, for all $\abs{y-x}\leq 30\sqrt{d}L_n$, $$P_{y,\omega}(X^*_{L_n^2}\geq v)\leq \exp\left(-\frac{v}{D_n}\right).$$\end{con}

As with Control \ref{Holder}, this control depends upon $x\in\mathbb{R}^d$, $\omega\in\Omega$ and $n\geq 0$.  It is not true, in general, that this type of localization control is available for all such triples $(x,\omega,n)$, but it is shown in \cite{SZ} that such controls are available for large $n$, with high probability, on a large portion of space.

We now introduce the primary probabilistic statement concerning Controls \ref{Holder} and \ref{localization}.  Notice that the event defined below does not include the control of traps described in \cite{SZ}, which play in important role in propagating Control \ref{Holder} in their arguments.  Since we simply use the H\"older control there obtained, we do not require a further use of their control of traps.

Consider, for each $x\in\mathbb{R}^d$, the event \begin{equation}\label{mainevent} \mathcal{B}_n(x)=\left\{\;\omega\in\Omega\;|\;\textrm{Controls \ref{Holder} and \ref{localization} hold for the triple}\;(x,\omega,n).\;\right\}.\end{equation}  Notice that, in view of (\ref{stationary}), for all $x\in\mathbb{R}^d$ and $n\geq 0$, \begin{equation}\label{mainevent1}\mathbb{P}(B_n(x))=\mathbb{P}(B_n(0)).\end{equation} It is therefore shown that the probability of the compliment of $B_n(0)$ approaches zero as $n$ tends to infinity.

\begin{thm}\label{induction}  Assume (\ref{steady}).  There exist $L_0$ and $c_0$ sufficiently large and $\eta_0>0$ sufficiently small such that, for each $n\geq 0$, $$\mathbb{P}\left(\Omega\setminus B_n(0)\right)\leq L_n^{-M_0}.$$\end{thm}

We henceforth fix the constants $L_0$, $c_0$ and $\eta_0$ appearing above.  \begin{equation}\label{constants} \textrm{Fix constants}\;L_0, c_0\;\textrm{and}\;\eta_0\;\textrm{satisfying (\ref{D_1}) and the hypothesis of Theorems \ref{effectivediffusivity} and \ref{induction}.}\end{equation}

We conclude this section with a few basic observations concerning Control \ref{Holder}, Control \ref{localization} and the H\"older norms introduced in (\ref{levelholder}).  Since Control \ref{Holder} cannot be expected to hold globally in space, it will be frequently necessary to introduce cutoff functions of the type appearing in (\ref{cutoff}).  The primary purpose of Control \ref{localization} is to bound the error we introduce, as seen in the following proposition.

\begin{prop}\label{local11}  Assume (\ref{steady}) and (\ref{constants}).  Fix $x\in\mathbb{R}^d$, $\omega\in\Omega$ and $n\geq 0$ and suppose that Control \ref{localization} is satisfied for the triple $(x,\omega,n)$.  For $f\in L^\infty(\mathbb{R}^d)$ satisfying $$\d\left(\Supp(f), B_{30\sqrt{d}L_n}(x)\right)\geq D_n+30\sqrt{d}L_n,$$ let $u(y,t)$ satisfy (\ref{review_eq}) with initial data $f(y)$.  Then, for each $\abs{y-x}\leq 30\sqrt{d}L_n$, $$\abs{u(y,L_n^2)}\leq \exp\left(-\frac{\d(\Supp(f),y)}{D_n}\right)\norm{f}_{L^\infty(\mathbb{R}^d)}.$$\end{prop}

\begin{proof}  The proof is immediate from the representation formula for the solution.  We have, for each $y\in\mathbb{R}^d$, $$u(y,L_n^2)=P_{y,\omega}\left(f(X_{L_n^2})\right).$$  Therefore, $$\abs{u(y,L_n^2)}\leq P_{y,\omega}\left(X_{L_n^2}^*\geq \d(\Supp(f),y)\right)\norm{f}_{L^\infty(\mathbb{R}^d)}.$$  Since $\d(\Supp(f), B_{30\sqrt{d}L_n}(x))\geq D_n+30\sqrt{d}L_n$, and since Control \ref{localization} is satisfied for the triple $(x,\omega,n)$, this implies that, for all $\abs{y-x}\leq 30\sqrt{d}L_n$, $$\abs{u(y,L_n^2)}\leq \exp\left(-\frac{\d(\Supp(f),y)}{D_n}\right)\norm{f}_{L^\infty(\mathbb{R}^d)},$$ which completes the argument.  \end{proof}

The following two elementary propositions will be used to extend Control \ref{Holder} to a larger portion of space.  The first is an elementary and well-known fact concerning the product of H\"older continuous functions.

\begin{prop}\label{prelim_product}  For each $n\geq 0$, for every $f,g\in\C^{0,\beta}(\mathbb{R}^d)$, $$\abs{fg}_n\leq \abs{f}_n\abs{g}_n.$$\end{prop}

\begin{proof}  Fix $n\geq 0$ and $f,g\in\C^{0,\beta}(\mathbb{R}^d)$.  For every $x,y\in\mathbb{R}^d$, the triangle inequality implies $$\abs{f(x)g(x)-f(y)g(y)}\leq \abs{f(x)}\abs{g(x)-g(y)}+\abs{g(y)}\abs{f(x)-f(y)}.$$  Therefore, $$\sup_{x\neq y}L_n^\beta\frac{\abs{f(x)g(x)-f(y)g(y)}}{{\abs{x-y}^\beta}}\leq \norm{f}_{L^\infty(\mathbb{R}^d)}\sup_{x\neq y}L_n^\beta\frac{\abs{g(x)-g(y)}}{\abs{x-y}^\beta}+\norm{g}_{L^\infty(\mathbb{R}^d)}\sup_{x\neq y}L_n^\beta\frac{\abs{f(x)-f(y)}}{\abs{x-y}^\beta}.$$  And, since $$\norm{fg}_{L^\infty(\mathbb{R}^d)}\leq \norm{f}_{L^\infty(\mathbb{R}^d)}\norm{g}_{L^\infty(\mathbb{R}^d)},$$ we conclude that $$\abs{fg}_n\leq\norm{f}_{L^\infty(\mathbb{R}^d)}\abs{g}_n+\norm{g}_{L^\infty(\mathbb{R}^d)}\sup_{x\neq y}L_n^\beta\frac{\abs{f(x)-f(y)}}{\abs{x-y}^\beta}\leq \abs{f}_n\abs{g}_n,$$ which completes the argument.  \end{proof}

The final proposition will play the most important role in extending Control \ref{Holder}.  The only observation is that the H\"older norms introduced in (\ref{levelholder}) occur at the length scale $L_n$.  Therefore, a function agreeing locally with H\"older continuous functions on scale $L_n$ must itself be globally H\"older continuous.

\begin{prop}\label{prelim_extension}  Let $I$ be an arbitrary index and $n\geq 0$.  If $f:\mathbb{R}^d\rightarrow\mathbb{R}$ and $\left\{g_i:\mathbb{R}^d\rightarrow\mathbb{R}\right\}_{i\in I}$ are such that, for a collection $\left\{x_i\right\}_{i\in I}\subset \mathbb{R}^d$,\begin{equation}\label{Holder2}f=g_i\;\;\textrm{on}\;\;B(x_i, 20\sqrt{d}L_n) \;\;\textrm{and}\;\; \Supp(f)\subset \bigcup_{i\in I}B(x_i,10\sqrt{d}L_n),\end{equation} then $$\abs{f}_n\leq 3\sup_{i\in I}\abs{g_i}_n.$$\end{prop}

\begin{proof}  In view of (\ref{Holder2}), for each $x\in\mathbb{R}^d$ there exists $j\in I$ such that $f(x)=g_j(x)$.  Therefore, \begin{equation}\label{Holder3} \abs{f(x)}=\abs{g_j(x)}\leq \sup_{i\in I}\abs{g_i}_n.\end{equation}  It remains to bound the H\"{o}lder semi-norm.

If $x,y\in\mathbb{R}^d$ satisfy $\abs{x-y}\geq L_n$, in view of (\ref{Holder2}), for $j,k\in I$ satisfying $f(x)=g_j(x)$ and $f(y)\leq g_k(y)$, \begin{equation}\label{Holder4} L_n^\beta\frac{\abs{f(x)-f(y)}}{\abs{x-y}^\beta}\leq \abs{g_j(x)-g_k(y)}\leq 2\sup_{i\in I}\abs{g_i}_n.\end{equation}  If $\abs{x-y}<L_n$, in view of (\ref{Holder2}), there exists $j\in I$ such that $x,y\in B(x_j,20\sqrt{d}L_n).$  Therefore, for this $j\in I$, \begin{equation}\label{Holder5} L_n^\beta\frac{\abs{f(x)-f(y)}}{\abs{x-y}^\beta}=L_n^\beta\frac{\abs{g_j(x)-g_j(y)}}{\abs{x-y}^\beta}\leq \abs{g_j}_n\leq\sup_{i\in I}\abs{g_i}_n.\end{equation}  The claim follows by combining (\ref{Holder3}), (\ref{Holder4}) and (\ref{Holder5}).  \end{proof}

\subsection{The Existence of an Invariant Measure}

We observe that, by using the transformation group (\ref{transgroup}), the ensemble of diffusion processes define dynamics on the probability space $\Omega$.  That is, for each path $X_t\in\C([0,\infty);\mathbb{R}^d)$ and $\omega\in\Omega$, we can consider the trajectory $$t\rightarrow\tau_{X_t}\omega.$$  And, for each $\tilde{f}\in L^\infty(\Omega)$, we can consider, for each $t\geq 0$, the map $f_t:\Omega\times\C([0,\infty);\mathbb{R}^d)\rightarrow\mathbb{R}$ defined, for each $\omega\in\Omega$ and $X_s\in\C([0,\infty);\mathbb{R}^d)$, by $$f_t(\omega,X_s)=\tilde{f}(\tau_{X_t}\omega).$$

We say that a probability measure $\pi$ on $\Omega$ is invariant if, for each $\tilde{f}\in L^\infty(\mathbb{R}^d)$, the law of the family $\left\{f_t\right\}_{t\geq 0}$ is constant in time with respect to the annealed, semi-direct product measure $\pi\ltimes P_{0,\omega}$.  This is equivalent to the statement that, for each measurable subset $E\in\mathcal{F}$, for each $t\geq 0$, \begin{equation}\label{p_invariant_1}\int_{\Omega} P_{0,\omega}\left(\tau_{X_t}\omega\in E\right)\;d\pi=\pi(E).\end{equation} And, in particular, if $\pi$ is an invariant measure then, for each $t\geq 0$, for each $f\in L^\infty(\Omega)$, the corresponding annealed expectations satisfy $$\int_{\Omega} P_{0,\omega}f(\tau_{X_t}\omega)\;d\pi=\int_{\Omega}f(\omega)\;d\pi.$$  In \cite{F1}, under assumptions identical to those of this paper, it was shown that there exists a unique probability measure on $(\Omega,\mathcal{F})$ satisfying (\ref{p_invariant_1}) which is mutually absolutely continuous with respect to $\mathbb{P}$.

\begin{thm}\label{p_invariant}  Assume (\ref{steady}) and (\ref{constants}).  There exists a unique probability measure $\pi$ on $(\Omega,\mathcal{F})$ satisfying (\ref{p_invariant_1}) and which is mutually absolutely continuous with respect to $\mathbb{P}$.  Furthermore, $\pi$ defines an ergodic probability measure with respect to the canonical Markov process on $\Omega$ defining (\ref{p_invariant_1}).\end{thm}

Henceforth, for every $f\in L^\infty(\mathbb{R}^d)$, we will write $\mathbb{E}_\pi(f)$ to denote the expectation of $f$ with respect to the measure $\pi$.  Precisely, for each $f\in L^\infty(\mathbb{R}^d)$, \begin{equation}\label{p_mea}\mathbb{E}_\pi(f)=\int_{\Omega}f\;d\pi.\end{equation}  Furthermore, observe that since $\pi$ is mutually absolutely continuous with respect to $\mathbb{P}$, there is no ambiguity with respect to subsets of full measure.

\section{The Almost Sure Control of Diffusivity}

The purpose of this section is to control, on a subset of full probability, the limiting diffusivity of the environment.  Precisely, we require to show that $$\limsup_{t\rightarrow\infty}\frac{1}{td}P_{0,\omega}\left(\abs{X_t}^2\right)<\infty,$$  and, indeed, by using Theorem \ref{effectivediffusivity} and Controls \ref{Holder} and \ref{localization}, we prove that, on a subset of full probability, \begin{equation}\label{d_statement}\lim_{t\rightarrow\infty}\frac{1}{td}P_{0,\omega}\left(\abs{X_t}^2\right)=\overline{\alpha}.\end{equation}

Before proceeding with the argument, we introduce some useful notation.  For each $n\geq 0$, for $C>0$ independent of $n$, let $\tilde{\chi}_n:\mathbb{R}^d\rightarrow\mathbb{R}$ denote a smooth cutoff function satisfying $0\leq \tilde{\chi}_n\leq 1$ with \begin{equation}\label{d_cutoff} \tilde{\chi}_n=1\;\;\textrm{on}\;\;\overline{B}_{3\tilde{D}_n},\;\;\tilde{\chi}_n=0\;\;\textrm{on}\;\;\mathbb{R}^d\setminus B_{4\tilde{D}_n},\;\;\norm{D\tilde{\chi}_n}_{L^\infty(\mathbb{R}^d)}\leq C/\tilde{D}_n\;\;\textrm{and}\;\;\norm{D^2\tilde{\chi}_n}_{L^\infty(\mathbb{R}^d)}\leq C/\tilde{D}_n^2.\end{equation}  Let $q:\mathbb{R}^d\rightarrow\mathbb{R}$ be defined by $q(x)=\abs{x}^2$ and define, for each $n\geq 0$, \begin{equation}\label{d_quadratic} q_n(x)=\tilde{\chi}_n(x)q(x)\;\;\textrm{on}\;\;\mathbb{R}^d.\end{equation}

In view of (\ref{i_mart}), we are interested in controlling, on a subset of full probability, the growth of the solution $u:\mathbb{R}^d\times[0,\infty)\times\Omega\rightarrow\mathbb{R}$ to the equation \begin{equation}\label{d_eq} \left\{\begin{array}{ll} u_t-\frac{1}{2}\tr(A(y,\omega)D^2u)+b(y,\omega)\cdot Du=0 & \textrm{on}\;\;\mathbb{R}^d\times(0,\infty), \\ u=q & \textrm{on}\;\;\mathbb{R}^d\times\left\{0\right\}.\end{array}\right.\end{equation}  However, because the initial data $q(x)=\abs{x}^2$ has unbounded H\"older norm, in order to effectively apply Control \ref{Holder} we will approximate the initial data using the functions $\left\{q_n\right\}_{n\geq 0}$ and bound the corresponding error using Control \ref{localization}.

Therefore, for each integer $n\geq 0$ and real number $s\geq 0$, for $f:\mathbb{R}^d\rightarrow\mathbb{R}$ measurable satisfying, for instance, quadratic growth, we define \begin{equation}\label{d_op} R_nf(x,\omega)=u(x,L_n^2,\omega)\;\;\textrm{and}\;\;R_sf(x,\omega)=u(x,s,\omega),\end{equation} for $u(x,t,\omega)$ the solution of (\ref{d_eq}) corresponding to initial data $f$ and $\omega\in\Omega$.  Furthermore, for each integer $n\geq 0$, we define $$\overline{R}_nf(x)=\overline{u}(x,L_n^2),$$ for $\overline{u}:\mathbb{R}^d\times[0,\infty)\rightarrow\mathbb{R}$ satisfying \begin{equation}\label{d_heq} \left\{\begin{array}{ll} \overline{u}_t-\frac{\alpha_n}{2}\Delta \overline{u}=0 & \textrm{on}\;\;\mathbb{R}^d\times(0,\infty), \\ \overline{u}=f & \textrm{on}\;\;\mathbb{R}^d\times\left\{0\right\}.\end{array}\right.\end{equation}  For each real number $s\geq 0$, define $$\overline{R}_{n,s}f(x)=\overline{u}(x,s),$$ for $\overline{u}:\mathbb{R}^d\times[0,\infty)\rightarrow\mathbb{R}^d$ satisfying \begin{equation}\label{d_hheq}\left\{\begin{array}{ll} \overline{u}_t-\frac{\alpha_n}{2}\Delta \overline{u}=0 & \textrm{on}\;\;\mathbb{R}^d\times(0,\infty), \\ \overline{u}=f & \textrm{on}\;\;\mathbb{R}^d\times\left\{0\right\}.\end{array}\right.\end{equation}  Finally, for each integer $n\geq 0$, define the difference operator $$S_nf(x,\omega)=R_nf(x,\omega)-\overline{R}_nf(x).$$  Observe that Control \ref{Holder} maybe restated in terms of the operators $S_n$, recalling the cutoff function $\chi_{n,x}$ from (\ref{cutoff1}).

\begin{con}\label{d_Holdercontrol}  Fix $x\in\mathbb{R}^d$, $\omega\in\Omega$ and $n\geq 0$.  For each $f\in\C^{0,\beta}(\mathbb{R}^d)$, $$\abs{\chi_{n,x}S_nf}_n\leq L_n^{-\delta}\abs{f}_n.$$\end{con}

The following proposition describes an elementary and well-known fact about the interaction between the kernels $\overline{R}_n$ and the rescaled H\"older norms appearing in (\ref{levelholder}).

\begin{prop}\label{d_contract}  Assume (\ref{steady}) and (\ref{constants}).  For each $n\geq 0$ and $f\in\C^{0,\beta}(\mathbb{R}^d)$, $$\abs{\overline{R}_nf}_n\leq\abs{f}_n.$$\end{prop}

\begin{proof}  Fix $n\geq 0$ and $f\in\C^{0,\beta}(\mathbb{R}^d)$.  For each $x\in\mathbb{R}^d$, $$\overline{R}_nf(x)=\int_{\mathbb{R}^d}(4\pi \alpha_n L_n^2)^{-d/2}e^{-\abs{y}^2/4\alpha_n L_n^2}f(y+x)\;dy.$$  Therefore, \begin{equation}\label{0_contract_1}\norm{\overline{R}_nf}_{L^\infty(\mathbb{R}^d)}\leq\norm{f}_{L^\infty(\mathbb{R}^d)}.\end{equation}  It remains to bound the H\"older semi-norm.

Fix elements $y\neq z$ of $\mathbb{R}^d$.  Then, $$\abs{\overline{R}_nf(y)-\overline{R}_nf(z)}=\left|\int_{\mathbb{R}^d}(4\pi \alpha_n L_n^2)^{-d/2}e^{-\abs{y}^2/4\alpha_n L_n^2}\left(f(y+x)-f(z+x)\right)\;dy\right|,$$ and, therefore, \begin{equation}\label{0_contract_2}\sup_{y\neq z}\frac{\abs{\overline{R}_nf(y)-\overline{R}_nf(z)}}{\abs{y-z}^\beta}\leq \sup_{y\neq z}\frac{\abs{f(y)-f(z)}}{\abs{y-z}^\beta}.\end{equation}  The claim follows from (\ref{0_contract_1}) and (\ref{0_contract_2}).  \end{proof}

We describe the localization properties of the kernels $\overline{R}_n$ as well.  Notice the role of Theorem \ref{effectivediffusivity}.

\begin{prop}\label{d_hlocalize} Assume (\ref{steady}) and (\ref{constants}).  For each $n\geq 0$ and $0\leq t\leq L_{n+1}^2$, for every $f\in L^\infty(\mathbb{R}^d)$, for $C>0$ independent of $n$, $$\sup_{x\in B_{2\tilde{D}_{n+1}}}\abs{\overline{R}_{n,t}(1-\tilde{\chi}_{n+1})f(x,\omega)}\leq C\norm{f}_{L^\infty(\mathbb{R}^d)}e^{-\tilde{\kappa}_{n+1}}.$$\end{prop}

\begin{proof}  Fix $n\geq 0$ and $0\leq t\leq L_{n+1}^2$.  Then, for each $f\in L^\infty(\mathbb{R}^d)$, for each $x\in B_{2\tilde{D}_{n+1}}$, $$\abs{\overline{R}_{n,t}(1-\tilde{\chi}_{n+1})f(x,\omega)}\leq \norm{f}_{L^\infty(\mathbb{R}^d)}\int_{\mathbb{R}^d\setminus B_{\tilde{D}_{n+1}}(x)}(4\pi\alpha_n t)^{-d/2}e^{-\abs{y-x}^2/4\alpha_n t}\;dy.$$  Therefore, since $0\leq t\leq L_{n+1}^2$, Theorem \ref{effectivediffusivity} implies that, for $C>0$ and $c>0$ independent of $n$, for each $f\in L^\infty(\mathbb{R}^d)$ and $x\in B_{2\tilde{D}_{n+1}}$, \begin{multline*}\abs{\overline{R}_{n,t}(1-\tilde{\chi}_{n+1})f(x,\omega)}\leq \norm{f}_{L^\infty(\mathbb{R}^d)}\int_{\mathbb{R}^d\setminus B_{c\tilde{\kappa}_{n+1}}}e^{-\abs{y}^2}\;dy \\ \leq C\norm{f}_{L^\infty(\mathbb{R}^d)}\tilde{\kappa}_{n+1}^{d-2}e^{-c\tilde{\kappa}_{n+1}^2}\leq C\norm{f}_{L^\infty(\mathbb{R}^d)}e^{-\tilde{\kappa}_{n+1}},\end{multline*} which, since $n\geq 0$ and $0\leq t\leq L_{n+1}^2$ were arbitrary, completes the argument.  \end{proof}

The following argument is virtually identical to Proposition \ref{d_hlocalize} and accounts for quadratic initial data.  Again, notice the role of Theorem \ref{effectivediffusivity}.

\begin{prop}\label{d_hquadratic}  Assume (\ref{steady}) and (\ref{constants}).  For each $n\geq 0$ and $0\leq t\leq L_{n+1}^2$, for $C>0$ independent of $n$, $$\abs{\overline{R}_{n,t}(1-\tilde{\chi}_{n+1})q(0)}\leq Ce^{-\tilde{\kappa}_{n+1}}.$$\end{prop}

\begin{proof}  Fix $n\geq 0$ and $0\leq t\leq L_{n+1}^2$.  We have $$\abs{\overline{R}_{n,t}(1-\tilde{\chi}_{n+1})q(0)}\leq \int_{\mathbb{R}^d\setminus B_{3\tilde{D}_{n+1}}}\abs{y}^2(4\pi\alpha_n t)^{-d/2}e^{-\abs{y-x}^2/4\alpha_n t}\;dy.$$  Therefore, since $0\leq t\leq L_{n+1}^2$, Theorem \ref{effectivediffusivity} implies that there exists $C>0$ and $c>0$ independent of $n$ such that $$\abs{\overline{R}_{n,t}(1-\tilde{\chi}_{n+1})q(0)}\leq CL_{n+1}^2\int_{\mathbb{R}^d\setminus B_{c\tilde{\kappa}_{n+1}}}\abs{y}^2e^{-\abs{y}^2}\;dy\leq CL_{n+1}^2\tilde{\kappa}_{n+1}^de^{-c\tilde{\kappa}_{n+1}^2}\leq Ce^{-\tilde{\kappa}_{n+1}},$$ which, since $n\geq 0$ and $0\leq t\leq L_{n+1}^2$ were arbitrary, completes the argument.  \end{proof}

In order to control the limiting quantity appearing in (\ref{d_statement}) for times of order $L_n^2$, it will be necessary to obtain Controls \ref{Holder} and \ref{localization} at levels sufficiently smaller than $n$.  This is necessary because Control \ref{Holder} does not provide an effective coupling over short periods in time.  We therefore fix an integer $\overline{m}$ satisfying \begin{equation}\label{d_overline} 3\leq(1+a)^{\overline{m}}<4.\end{equation}  The existence of such an integer is guaranteed, in view of (\ref{Holderexponent}).  It is also necessary to obtain Controls \ref{Holder} and \ref{localization} on a large portion of space.

For each $n\geq 0$, let $$\tilde{A}_n=\left\{\;\omega\in\Omega\;|\;\textrm{We have Controls}\;\ref{Holder}\;\textrm{and}\;\ref{localization}\;\textrm{for each}\;x\in L_{n}\mathbb{Z}^d\cap B_{5\tilde{D}_{n+1}}\;\right\},$$ and, for each $n\geq \overline{m}$, let \begin{equation}\label{d_prevent} A_n=\bigcap_{j=n-\overline{m}}^{n+1}\tilde{A}_j.\end{equation}  The following proposition proves that the conditions of $A_n$ are satisfied with high probability.

\begin{prop}\label{d_probability}  Assume (\ref{steady}) and (\ref{constants}).  For each $n\geq \overline{m}$, for $C>0$ independent of $n$, $$\mathbb{P}\left(\Omega\setminus A_n\right)\leq CL_{n-\overline{m}}^{(d+1)a-M_0}.$$\end{prop}

\begin{proof}  In view of (\ref{mainevent1}) and Theorem \ref{induction}, for each $n\geq 0$, for $C>0$ independent of $n$, $$\mathbb{P}\left(\Omega\setminus \tilde{A}_n\right)\leq C\left(\frac{\tilde{D}_{n+1}}{L_n}\right)^dL_n^{-M_0}.$$  And, using (\ref{L}), (\ref{kappa}), (\ref{D}) and (\ref{delta}), since there exists $C>0$ satisfying, for each $n\geq 0$, $$\tilde{\kappa}_n^d\leq CL_n^a,$$ we have, for $C>0$ independent of $n$, $$\mathbb{P}\left(\Omega\setminus \tilde{A}_n\right)\leq C\tilde{\kappa}_{n+1}^dL_n^{da-M_0}\leq CL_n^{(d+1)a-M_0}.$$  Therefore, since $(d+1)a-M_0<0$, for each $n\geq \overline{m}$, for $C>0$ independent of $n$, $$\mathbb{P}\left(\Omega\setminus A_n\right)\leq \sum_{j=n-\overline{m}}^{n+1}\mathbb{P}\left(\Omega\setminus \tilde{A}_j\right)\leq CL_{n-\overline{m}}^{(d+1)a-M_0},$$ which completes the argument.  \end{proof}

The following proposition estimates the error introduced by localizing the quadratic initial data.  We prove that, on the event $\tilde{A}_{n+1}$, this error vanishes as $n$ approaches infinity.

\begin{prop}\label{d_localize}  Assume (\ref{steady}).  For each $n\geq 0$ and $\omega\in \tilde{A}_{n+1}$, for all $0\leq t<L_{n+1}^2$, $$\sup_{x\in B_{2\tilde{D}_{n+1}}}P_{x,\omega}\left((1-\tilde{\chi}_{n+1}(X_t))\abs{X_t}^2\right)\leq C\tilde{D}_{n+1}^2e^{-\kappa_{n+1}}.$$\end{prop}

\begin{proof}  Fix $n\geq 0$, $\omega\in \tilde{A}_{n+1}$ and $L_n^2\leq t\leq L_{n+1}^2$.  Then, recalling (\ref{star}), and in view of (\ref{d_cutoff}), for every $x\in B_{2\tilde{D}_{n+1}}$, since $0\leq t\leq L_{n+1}^2$, $$P_{x,\omega}\left((1-\tilde{\chi}_{n+1}(X_t))\abs{X_t}^2\right)\leq \tilde{D}_{n+1}^2P_{x,\omega}\left(X^*_{L_{n+1}^2}\geq \tilde{D}_{n+1}\right)+2\int_{\tilde{D}_{n+1}}^\infty rP_{x,\omega}\left(X^*_{L_{n+1}^2}\geq r\right)\;dr.$$  Therefore, since $\omega\in \tilde{A}_{n+1}$, Control \ref{localization} implies that, for $C>0$ independent of $n$, $$\sup_{x\in B_{2\tilde{D}_{n+1}}}P_{x,\omega}\left((1-\tilde{\chi}_{n+1}(X_t))\abs{X_t}^2\right)\leq C\tilde{D}_{n+1}^2e^{-\kappa_{n+1}},$$ which, since $n\geq 0$, $\omega\in \tilde{A}_{n+1}$ and $0\leq t\leq L_{n+1}^2$ were arbitrary, completes the argument.  \end{proof}

We are now prepared to present our primary control of the diffusivity.

\begin{prop}\label{d_main} Assume (\ref{steady}) and (\ref{constants}).  For each $n\geq \overline{m}$, $\omega\in A_n$ and $L_n^2\leq t< L_{n+1}^2$, for $C>0$ independent of $n$ and $t$, $$\abs{\frac{1}{td}R_tq(0,\omega)-\overline{\alpha}}\leq \max_{0\leq j\leq \overline{m}}\left\{\abs{\alpha_{n-j}-\overline\alpha}\right\}+CL_{n-\overline{m}}^{10a-\delta}.$$\end{prop}

\begin{proof}  Fix $n\geq \overline{m}$, $\omega\in A_n$ and $L_n^2\leq t\leq L_{n+1}^2$.  In what follows, we suppress the dependence on $\omega$ in our notation.

We first recall (\ref{d_quadratic}) and observe using Proposition \ref{d_localize} that, since $L_n^2\leq t<L_{n+1}^2$, for $C>0$ independent of $n$ and $t$, \begin{equation}\label{d_main_0}\abs{R_tq(0,\omega)-R_tq_{n+1}(0,\omega)}=P_{0,\omega}\left((1-\tilde{\chi}_{n+1}(X_t))\abs{X_t}^2\right)\leq C\tilde{D}_{n+1}^2e^{-\kappa_{n+1}}.\end{equation}  And, recalling (\ref{levelholder}), using (\ref{d_cutoff}), for each $m\leq n$, for $C>0$ independent of $m$ and $n$, \begin{equation}\label{d_main_00} \abs{q_{n+1}}_m\leq \abs{q_{n+1}}_n\leq C\tilde{D}_{n+1}^2.\end{equation}  This estimate allows us to effectively apply Control \ref{Holder}.

We now decompose the operator $R_t$.  Recalling (\ref{L}), choose $1\leq k_n<\ell_n^2$ such that \begin{equation}\label{d_main_1} k_nL_n^2\leq t<(k_n+1)L_n^2,\end{equation} and, proceeding inductively, for each $1\leq j\leq \overline{m}$, define  $1\leq k_j<\ell_j^2$ such that \begin{equation}\label{d_main_2} k_{n-j}L_{n-j}^2\leq t-\sum_{i=0}^{j-1}k_{n-i}L_{n-i}^2<(k_{n-j}+1)L_{n-j}^2.\end{equation}  We write \begin{equation}\label{d_main_3} \tilde{t}=t-\sum_{j=0}^{\overline{m}}k_{n-j}L_{n-j}^2,\end{equation} observing that $0\leq \tilde{t}<L_{n-\overline{m}}^2$, and form the decomposition \begin{equation}\label{d_main_4}R_t=R_{\tilde{t}}\prod_{j=0}^{\overline{m}}\left(R_{n-j}\right)^{k_{n-j}}.\end{equation}

In order to apply Control \ref{Holder}, it remains to localize the kernels appearing in (\ref{d_main_4}).  In what follows, we will use that fact that, for each $s\geq 0$, $$\norm{R_sq_{n+1}}_{L^\infty(\mathbb{R}^d)}\leq \norm{q_{n+1}}_{L^\infty(\mathbb{R}^d)}.$$  We observe using Control \ref{localization} and (\ref{d_main_00}) that, for each $x\in B_{2\tilde{D}_{n+1}}$, since $\omega\in A_n$ and $t<L_{n+1}^2$, for $C>0$ independent of $n$, $$\abs{R_{\tilde{t}}\prod_{j=0}^{\overline{m}}\left(R_{n-j}\right)^{k_{n-j}}q_{n+1}(x)-R_{\tilde{t}}\prod_{j=1}^{\overline{m}}\left(R_{n-j}\right)^{k_{n-j}}\left(R_n\right)^{k_n-1}\tilde{\chi}_{n+1}R_nq_{n+1}(x)}\leq C\tilde{D}_{n+1}^2 e^{-\kappa_{n+1}}.$$  Proceeding inductively, for $C>0$ independent of $n$, for each $x\in B_{2\tilde{D}_{n+1}}$, \begin{equation}\label{d_main_6} \abs{R_tq_{n+1}(x)-R_{\tilde{t}}\prod_{j=1}^{\overline{m}}\left(R_{n-j}\right)^{k_{n-j}}\left(\tilde{\chi}_{n+1}R_n\right)^{k_n}q_{n+1}(x)}\leq Ck_n\tilde{D}_{n+1}^2e^{-\kappa_{n+1}}\leq C\ell_n^2\tilde{D}_{n+1}^2e^{-\kappa_{n+1}}.\end{equation}

In an identical fashion, we use Control \ref{localization} and (\ref{d_main_00}) to conclude that, since $\omega\in A_n$ and since, for each $0\leq j\leq \overline{m}$, $$\tilde{t}+\sum_{i=j}^{\overline{m}}k_{n-i}L_{n-i}^2 < L_{n-j+1}^2,$$ for each $x\in B_{2\tilde{D}_{n-\overline{m}+1}}$, for $C>0$ independent of $n$, \begin{multline}\label{d_main_7} \abs{R_tq_{n+1}(x)-R_{\tilde{t}}\prod_{j=0}^{\overline{m}}\left(\tilde{\chi}_{n-j+1}R_{n-j}\right)^{k_{n-j}}q_{n+1}(x)}\leq C\sum_{j=0}^{\overline{m}}\ell_{n-j}^2\tilde{D}_{n+1}^2e^{-\kappa_{n-j+1}} \\ \leq C\ell_{n-\overline{m}}^2\tilde{D}_{n+1}^2e^{-\kappa_{n-\overline{m}+1}}.\end{multline}

We are now prepared to apply Control \ref{Holder}.  Notice that, for each $x\in\mathbb{R}^d$,  for integers $s_i\geq 0$, \begin{multline*}\left(\tilde{\chi}_{n+1}R_n\right)^{k_n}q_{n+1}(x) = \left(\tilde{\chi}_{n+1}S_n+\tilde{\chi}_{n+1}\overline{R}_n\right)^{k_n} =\left(\overline{R}_n\right)^{k_n}q_{n+1}(x) \\ +\sum_{m=1}^{k_n}\sum_{s_0+s_1+\ldots+s_m+m=k_n}\left(\tilde{\chi}_{n+1}\overline{R}_n\right)^{s_0}\tilde{\chi}_{n+1}S_n\left(\tilde{\chi}_{n+1}\overline{R}_n\right)^{s_1}\ldots\tilde{\chi}_{n+1}S_n\left(\overline{R}_n\right)^{s_m}q_{n+1}(x).\end{multline*}  Therefore, using Proposition \ref{prelim_product}, Proposition \ref{prelim_extension} and Proposition \ref{d_contract}, since $\omega\in A_n$, Control \ref{Holder} implies that, for each $x\in\mathbb{R}^d$, for $C>0$ independent of $n$, \begin{multline}\label{d_main_8} \left|\left(\tilde{\chi}_{n+1}R_n\right)^{k_n}q_{n+1}(x)-\left(\tilde{\chi}_{n+1}\overline{R}_n\right)^{k_n}q_{n+1}(x)\right|\leq C\sum_{m=1}^{k_n} {k_n \choose m}3^mL_n^{-m\delta}\tilde{D}_{n+1}^2 \\ \leq C\sum_{m=1}^{\ell_n^2}{\ell_n^2 \choose m} 3^m L_n^{-m\delta}\tilde{D}_{n+1}^2\leq C\sum_{m=1}^{\ell_n^2}\frac{3^m}{m!}L_n^{m(2a-\delta)}\tilde{D}_{n+1}^2\leq CL_n^{2a-\delta}\tilde{D}_{n+1}^2.\end{multline}

Therefore, for each $x\in\mathbb{R}^d$, for $C>0$ independent of $n$, $$\left| R_{\tilde{t}}\prod_{j=0}^{\overline{m}}\left(\tilde{\chi}_{n-j+1}R_{n-j}\right)^{k_{n-j}}q_{n+1}(x)-R_{\tilde{t}}\prod_{j=1}^{\overline{m}}\left(\tilde{\chi}_{n-j+1}R_{n-j}\right)^{k_{n-j}}\left(\tilde{\chi}_{n+1}\overline{R}_n\right)^{k_n}q_{n+1}(x)\right|\leq CL_n^{2a-\delta}\tilde{D}_{n+1}^2.$$  And, proceeding by an identical argument, since Proposition \ref{prelim_product}, Proposition \ref{d_contract} and (\ref{d_main_00}) imply that, for each $0\leq j\leq \overline{m}$, $$\abs{\prod_{i=0}^j\left(\tilde{\chi}_{n-i+1}\overline{R}_{n-i}\right)^{k_{n-i}}q_{n+1}(x)}_{n-j-1}\leq \abs{q_{n+1}}_{n-j-1}\leq C\tilde{D}_{n+1}^2,$$ we have, for each $x\in\mathbb{R}^d$, for $C>0$ independent of $n$, \begin{multline}\label{d_main_9}  \left| R_{\tilde{t}}\prod_{j=0}^{\overline{m}}\left(\tilde{\chi}_{n-j+1}R_{n-j}\right)^{k_{n-j}}q_{n+1}(x)-R_{\tilde{t}}\prod_{j=0}^{\overline{m}}\left(\tilde{\chi}_{n-j+1}\overline{R}_{n-j}\right)^{k_{n-j}}q_{n+1}(x)\right| \\ \leq \sum_{j=0}^{\overline{m}} CL_{n-j}^{2a-\delta}\tilde{D}_{n+1}^2\leq CL_{n-\overline{m}}^{2a-\delta}\tilde{D}_{n+1}^2.\end{multline}

In what follows, we use that fact that, for each $n\geq 0$ and $t\geq 0$, for every $f\in L^\infty(\mathbb{R}^d)$, $$\norm{\overline{R}_{n,t}f}_{L^\infty(\mathbb{R}^d)}\leq \norm{f}_{L^\infty(\mathbb{R}^d)}.$$  Proposition \ref{d_hlocalize} and (\ref{d_main_00}) imply that, for each $x\in B_{2\tilde{D}_{n+1}}$, for $C>0$ independent of $n$, \begin{multline*}\left|\prod_{j=0}^{\overline{m}}\left(\tilde{\chi}_{n-j+1}\overline{R}_{n-j}\right)^{k_{n-j}}q_{n+1}(x)-\prod_{j=1}^{\overline{m}}\left(\tilde{\chi}_{n-j+1}\overline{R}_{n-j}\right)^{k_{n-j}}\left(\tilde{\chi}_{n+1}\overline{R}_n\right)^{k_n-1}\overline{R}_nq_{n+1}(x)\right| \\ \leq C\tilde{D}_{n+1}^2e^{-\tilde{\kappa}_{n+1}}.\end{multline*}  And, proceeding inductively, for $C>0$ independent of $n$, for each $x\in B_{2\tilde{D}_{n+1}}$, \begin{multline*}\left|\prod_{j=0}^{\overline{m}}\left(\tilde{\chi}_{n-j+1}\overline{R}_{n-j}\right)^{k_{n-j}}q_{n+1}(x)-\prod_{j=1}^{\overline{m}}\left(\tilde{\chi}_{n-j+1}\overline{R}_{n-j}\right)^{k_{n-j}}\left(\overline{R}_n\right)^{k_n}q_{n+1}(x)\right| \\ \leq Ck_n\tilde{D}_{n+1}^2e^{-\tilde{\kappa}_{n+1}}\leq C\ell_n^2\tilde{D}_{n+1}^2e^{-\tilde{\kappa}_{n+1}}.\end{multline*}  By repeating the identical argument, for $C>0$ independent of $n$, for each $x\in B_{2\tilde{D}_{n-\overline{m}+1}}$, \begin{multline}\label{d_main_10}\left|\prod_{j=0}^{\overline{m}}\left(\tilde{\chi}_{n-j+1}\overline{R}_{n-j}\right)^{k_{n-j}}q_{n+1}(x)-\prod_{j=0}^{\overline{m}}\left(\overline{R}_{n-j}\right)^{k_{n-j}}q_{n+1}(x)\right| \\ \leq C\sum_{j=0}^{\overline{m}}\ell_{n-j}^2\tilde{D}_{n+1}^2e^{-\tilde{\kappa}_{n-j+1}}\leq C\ell_{n-\overline{m}}^2\tilde{D}_{n+1}^2e^{-\tilde{\kappa}_{n-\overline{m}+1}}.\end{multline}

Recalling $\chi_{\tilde{D}_{n-\overline{m}}}$ defined in (\ref{cutoff}), we observe that \begin{multline*}\left| R_{\tilde{t}}\left(\prod_{j=0}^{\overline{m}}\left(\tilde{\chi}_{n-j+1}\overline{R}_{n-j}\right)^{k_{n-j}}-\prod_{j=0}^{\overline{m}}\left(\overline{R}_{n-j}\right)^{k_{n-j}}\right)q_{n+1}(0)\right| \leq  \\ \left| R_{\tilde{t}}\chi_{\tilde{D}_{n-\overline{m}}}\left(\prod_{j=0}^{\overline{m}}\left(\tilde{\chi}_{n-j+1}\overline{R}_{n-j}\right)^{k_{n-j}}-\prod_{j=0}^{\overline{m}}\left(\overline{R}_{n-j}\right)^{k_{n-j}}\right)q_{n+1}(0)\right| \\ + \left| R_{\tilde{t}}(1-\chi_{\tilde{D}_{n-\overline{m}}})\left(\prod_{j=0}^{\overline{m}}\left(\tilde{\chi}_{n-j+1}\overline{R}_{n-j}\right)^{k_{n-j}}-\prod_{j=0}^{\overline{m}}\left(\overline{R}_{n-j}\right)^{k_{n-j}}\right)q_{n+1}(0)\right|.\end{multline*}  Therefore, since $\omega\in A_n$ and $0\leq \tilde{t}< L_{n-\overline{m}}^2$, Control \ref{localization} and (\ref{d_main_10}) imply that, for $C>0$ independent of $n$, \begin{multline}\label{d_main_11} \left| R_{\tilde{t}}\left(\prod_{j=0}^{\overline{m}}\left(\tilde{\chi}_{n-j+1}\overline{R}_{n-j}\right)^{k_{n-j}}-\prod_{j=0}^{\overline{m}}\left(\overline{R}_{n-j}\right)^{k_{n-j}}\right)q_{n+1}(0)\right| \\ \leq C\ell_{n-\overline{m}}^2\tilde{D}_{n+1}^2e^{-\tilde{\kappa}_{n-\overline{m}+1}}+C\tilde{D}_{n+1}^2e^{-\tilde{\kappa}_{n-\overline{m}}}\leq C\ell_{n-\overline{m}}^2\tilde{D}_{n+1}^2e^{-\tilde{\kappa}_{n-\overline{m}}}.\end{multline}

Since, for $C>0$ independent of $n$, $$\norm{D\left(\prod_{j=0}^{\overline{m}}\left(\overline{R}_{n-j}\right)^{k_{n-j}}q_{n+1}\right)}_{L^\infty(\mathbb{R}^d)}\leq C\tilde{D}_{n+1},$$ and $$\norm{D^2\left(\prod_{j=0}^{\overline{m}}\left(\overline{R}_{n-j}\right)^{k_{n-j}}q_{n+1}\right)}_{L^\infty(\mathbb{R}^d)}\leq C,$$ we have, using the comparison principle, for every $x\in\mathbb{R}^d$, for $C>0$ independent of $n$, \begin{equation}\label{d_main_12} \left|(R_{\tilde{t}}-\overline{R}_{n,\tilde{t}})\prod_{j=0}^{\overline{m}}\left(\overline{R}_{n-j}\right)^{k_{n-j}}q_{n+1}\right|\leq C\tilde{t}\tilde{D}_{n+1}.\end{equation}  Finally, by a small modification to the argument appearing in Proposition \ref{d_hquadratic}, since $t\leq L_{n+1}^2$, for $C>0$ independent of $n$, \begin{equation}\label{d_main_14} \left|\overline{R}_{n,\tilde{t}}\prod_{j=0}^{\overline{m}}\left(\overline{R}_{n-j}\right)^{k_{n-j}}(q_{n+1}-q)(0)\right|\leq Ce^{-\tilde{\kappa}_{n+1}}.\end{equation}

We therefore conclude that, using (\ref{d_main_0}), (\ref{d_main_6}), (\ref{d_main_9}), (\ref{d_main_10}), (\ref{d_main_11}), (\ref{d_main_12}) and (\ref{d_main_14}), for $C>0$ independent of $n$, \begin{multline}\left|R_tq(0,\omega)-\overline{R}_{n,\tilde{t}}\prod_{j=0}^{\overline{m}}\left(\overline{R}_{n-j}\right)^{k_{n-j}}q(0)\right| \\ =\left|R_tq(0,\omega)-d\left(\alpha_n\tilde{t}+\sum_{j=0}^{\overline{m}}k_{n-j}L_{n-j}^2\alpha_{n-j}\right)\right|\leq CL_{n-\overline{m}}^{2a-\delta}\tilde{D}_{n+1}^2+C\tilde{t}\tilde{D}_{n+1}.\end{multline}  Here, we use that fact that, excepting (\ref{d_main_9}) and (\ref{d_main_12}), the errors appearing on the righthand sides of (\ref{d_main_0}), (\ref{d_main_6}), (\ref{d_main_10}), (\ref{d_main_11}) and (\ref{d_main_14}) vanish as $n$ approaches infinity.

Therefore, using (\ref{Holderexponent}), (\ref{L}), (\ref{kappa}), (\ref{D}), (\ref{delta}) and (\ref{d_overline}), since $t\geq L_n^2$ and $0\leq \tilde{t}<L_{n-\overline{m}}^2$, for $C>0$ independent of $n$, \begin{multline*}\abs{\frac{1}{td}R_tq(0)-\overline{\alpha}}\leq \max_{0\leq j\leq \overline{m}}\left\{\abs{\alpha_{n-j}-\overline\alpha}\right\}+C\tilde{\kappa}_{n+1}^2\left(L_{n-\overline{m}}^{2a(1+a)^{\overline{m}}+2a-\delta}+L_{n-\overline{m}}^{2-(1-a)(1+a)^{\overline{m}}}\right) \\ \leq\max_{0\leq j\leq \overline{m}}\left\{\abs{\alpha_{n-j}-\overline\alpha}\right\}+CL_{n-\overline{m}}^{10a-\delta},\end{multline*} which, since $n\geq 0$, $\omega\in A_n$ and $L_n^2\leq t<L_{n+1}^2$ were arbitrary, completes the argument.  \end{proof}

We may now define the subset of full probability on which we obtain (\ref{d_statement}).  In view of (\ref{Holderexponent}), (\ref{L}), (\ref{delta}) and Proposition \ref{d_probability}, $$\sum_{n=\overline{m}}^\infty \mathbb{P}\left(\Omega\setminus A_n\right)\leq C\sum_{n=\overline{m}}^\infty L_{n-\overline{m}}^{(d+1)a-M_0}<\infty.$$  Therefore, using the Borel-Catelli lemma, we define the subset of full probability \begin{equation}\label{d_subset}\Omega_1=\left\{\;\omega\in \Omega\;|\;\textrm{There exists}\;\overline{n}(\omega)\geq\overline{m}\;\textrm{such that}\;\omega\in A_n\;\textrm{for all}\;n\geq\overline{n}.\;\right\}.\end{equation}  We conclude with the primary result of this section.

\begin{thm}\label{d_diffusive} Assume (\ref{steady}) and (\ref{constants}).  For each $\omega\in \Omega_1$, $$\lim_{t\rightarrow\infty} \frac{1}{td}P_{0,\omega}\left(\abs{X_t}^2\right)=\overline{\alpha}.$$  In particular, for each $\omega\in\Omega_1$ there exists $C_1(\omega)>0$ satisfying $$\sup_{t\geq 1}\frac{1}{td}P_{0,\omega}\left(\abs{X_t}^2\right)<C_1.$$\end{thm}

\begin{proof}  The representation formula for the solution of (\ref{d_eq}) yields $$R_tq(0,\omega)=P_{0,\omega}\left(\abs{X_t}^2\right).$$  Therefore, the proof of convergence follows immediately from the definition of $\Omega_1$ and Proposition \ref{d_main}, since Theorem \ref{effectivediffusivity} implies that $$\lim_{n\rightarrow\infty}\max_{0\leq j\leq \overline{m}}\left\{\abs{\alpha_{n-j}-\overline\alpha}\right\}=0,$$ and (\ref{Holderexponent}), (\ref{L}) and (\ref{delta}) imply that $$\lim_{n\rightarrow\infty}L_{n-\overline{m}}^{10a-\delta}=0.$$  The corresponding bound then follows from the convergence and the fact that solutions to (\ref{d_eq}) with quadratic initial data are continuous.  \end{proof}

\section{A General Entropy Estimate}

In this section, we prove that, on a subset of full probability, the physical entropy determined by the Green's functions, as defined, for each $x\in\mathbb{R}^d$, $t\geq 0$ and $\omega\in\Omega$, by $$H_{t,\omega}(x)=\int_{\mathbb{R}^d}-p_{t,\omega}(x,y)\log\left(p_{t,\omega}(x,y)\right)\;dy,$$ grows at most logarithmically in time.  We remark that the results of this section do not rely upon our fine assumptions regarding the coefficients.  The finite range dependence, isotropy and stationarity are not used here.  Instead, we require only the boundedness, Lipschitz continuity and ellipiticity.

Before proceeding with the proof, we recall some basic facts about the Green's functions which hold independently of $\omega\in\Omega$.  In view of \cite{Fr}, using (\ref{bounded}), (\ref{Lipschitz}) and (\ref{elliptic}), for each $\omega\in \Omega$, the Green's function is such that, for each $T>0$ there exist $C=C(T)>0$ and $c(T)>0$, independent of $\omega$, satisfying, for each $0<t\leq T$, \begin{equation}\label{e_green} \abs{p_{t,\omega}(x,y)}\leq Ct^{-d/2}e^{-c\abs{x-y}^2/t}\;\;\textrm{and}\;\;\abs{D_xp_{t,\omega}(x,y)}\leq Ct^{-(d+1)/2}e^{-c\abs{x-y}^2/t}.\end{equation}  And, for each $t\geq 0$, for each continuous $g:\mathbb{R}^d\rightarrow\mathbb{R}$ growing, for instance, at most quadratically, \begin{equation}\label{e_green_1}R_tg(x,\omega)=P_{x,\omega}\left(g(X_t)\right)=\int_{\mathbb{R}^d}p_{t,\omega}(x,y)g(y)\;dy.\end{equation}

In the following proposition we obtain a rough localization estimate for the Green's functions.  Notice that although this estimate holds globally for $x\in\mathbb{R}^d$ and $\omega\in\Omega$, it only provides an effective estimate at length scales which are much larger than those appearing in Control \ref{localization}.  Recall the notation (\ref{star}).

\begin{prop}\label{e_localize}  Assume (\ref{steady}).  For each $\omega\in\Omega$, $x\in\mathbb{R}^d$ and $t\geq 0$, for $C_1\geq 1$ independent of $x$, $\omega$ and $t$, for each $R>0$, $$P_{x,\omega}\left(X^*_t\geq R\right)\leq e^{-\frac{(R-C_1t)_+^2}{C_1t}}.$$ \end{prop}

\begin{proof}  Fix $x\in\mathbb{R}^d$, $\omega\in\Omega$, $t\geq 0$ and $R\geq 0$.  We recall that, almost surely with respect to $P_{x,\omega}$, for  $B_s$ a Brownian motion on $\mathbb{R}^d$ under $P_{x,\omega}$ with respect to the canonical right-continuous filtration on $\C([0,\infty);\mathbb{R}^d)$, paths $X_s\in\C([0,\infty);\mathbb{R}^d)$ satisfy the stochastic differential equation $$\left\{\begin{array}{l} dX_s=-b(X_s,\omega)dt+\sigma(X_s,\omega)dB_s, \\ X_0=x.\end{array}\right.$$  Therefore, using the exponential inequality for Martingales, see Revuz and Yor \cite{RY}, (\ref{bounded}) and (\ref{Lipschitz}), for every $\tilde{R}\geq 0$, for $C_1>0$ independent of $\tilde{R}$, $t$, $x$ and $\omega$, \begin{equation}\label{u_martingale_2}P_{x,\omega}\left(X^*_t\geq \tilde{R}+Ct\right)\leq e^{-\frac{\tilde{R}^2}{Ct}}.\end{equation}

Therefore, by choosing $\tilde{R}=(R-CT)_+$ in (\ref{u_martingale_2}), for $C>0$ independent of $x$, $t$, $\omega$ and $R$, $$P_{x,\omega}\left(X^*_t\geq R\right)\leq e^{-\frac{(R-Ct)_+^2}{Ct}},$$ which, since $x$, $t$, $\omega$ and $R$ were arbitrary, completes the argument.  \end{proof}

The following proposition provides our basic control of the entropy.  We prove that, uniformly in $\Omega$, the entropy $H_{t,\omega}(0)$ grows at most logarithmically in time.

\begin{prop}\label{e_log}  Assume (\ref{steady}).  For each $\omega\in\Omega$, there exists $C>0$ independent of $\omega$ such that, for each $t\geq 1$, $$H_{t,\omega}(0)\leq C\left(\log(t)+1\right).$$\end{prop}

\begin{proof}  Fix $\omega\in\Omega$ and $t\geq 2$.  We will use the fact that, for each $z\in\mathbb{R}^d$, $$p_{t,\omega}(0,z)=\int_{\mathbb{R}^d}p_{t-1,\omega}(0,y)p_{1,\omega}(y,z)\;dy.$$  For each $z\in\mathbb{R}^d$, $$p_{t,\omega}(0,z)=\int_{B_{\abs{z}/2}(z)}p_{t-1,\omega}(0,y)p_{1,\omega}(y,z)\;dy+\int_{\mathbb{R}^d\setminus B_{\abs{z}/2}(z)}p_{t-1,\omega}(0,y)p_{1,\omega}(y,z)\;dy,$$ and, therefore, for each $z\in\mathbb{R}^d$, \begin{multline*}0< p_{t,\omega}(0,z) \leq P_{0,\omega}\left(X^*_{t-1}\geq \abs{z}/2\right)\norm{p_{1,\omega}(y,z)}_{L^\infty(\mathbb{R}^d\times\mathbb{R}^d)} \\ +\norm{p_{1,\omega}(y,z)}_{L^\infty(\mathbb{R}^d\setminus B_{\abs{z}/2}(z))}\norm{p_{t-1,\omega}(0,y)}_{L^1(\mathbb{R}^d)}.\end{multline*}  Therefore, using (\ref{e_green}) and Proposition \ref{e_localize}, for $C_1>0$ as in Proposition \ref{e_localize}, for each $\abs{z}\geq 2C_1t^2$, for $C_2>0$ and $c_2>0$ independent of $\omega$ and $t\geq 2$, \begin{equation}\label{e_log_1}0< p_{t,\omega}(0,z)\leq C\left(e^{-\frac{(\abs{z}/2-C_1(t-1))^2_+}{C_1(t-1)}}+e^{-c\abs{z}^2/4}\right)\leq C_2e^{-c_2\abs{z}}.\end{equation}

Therefore, we write \begin{equation}\label{e_log_0}H_{t,\omega}(0)=\int_{B_{2C_1t^2}}-p_{t,\omega}(0,y)\log\left(p_{t,\omega}(0,y)\right)\;dy+\int_{\mathbb{R}^d\setminus B_{2C_1t^2}}-p_{t,\omega}(0,y)\log\left(p_{t,\omega}(0,y)\right)\;dy.\end{equation}  Using Jensen's inequality and the fact that the function $-x\log(x)$ is concave, $$\int_{B_{2C_1t^2}}-p_{t,\omega}(0,y)\log\left(p_{t,\omega}(0,y)\right)\;dy\leq -\int_{B_{2C_1t^2}}p_{t,\omega}(0,y)\;dy\log\left(\frac{1}{\abs{B_{2C_1t^2}}}\int_{B_{2C_1t^2}}p_{t,\omega}(0,y)\;dy\right).$$  Therefore, in view of Proposition \ref{e_localize}, there exists $C>0$ satisfying, for each $t\geq 2$, \begin{equation}\label{e_log_2}\int_{B_{2C_1t^2}}-p_{t,\omega}(0,y)\log\left(p_{t,\omega}(0,y)\right)\;dy\leq \log\left(C\abs{B_{2C_1t^2}}\right)\leq C\left(\log(t)+1\right).\end{equation}

Furthermore, there exists $\overline{t}\geq 2$ such that, whenever $t\geq \overline{t}$ and $\abs{z}\geq 2C_1t^2$, $$C_2e^{-c_2\abs{z}}\leq \frac{1}{e}.$$  Here, observe that the function $-x\log(x)$ achieves a maximum of $1/e$ at $x=1/e$.  We therefore conclude that, whenever $t\geq\overline{t}$, for $C>0$ independent of $t\geq\overline{t}$ and $\omega$, \begin{equation}\label{e_log_3}\int_{\mathbb{R}^d\setminus B_{2C_1t^2}}-p_{t,\omega}(0,y)\log\left(p_{t,\omega}(0,y)\right)\;dy\leq \int_{\mathbb{R}^d\setminus B_{2C_1t^2}}C_2\exp^{-c_2\abs{y}}\left(-\log\left(C_2\right)+c_2\abs{y}\right)\;dy\leq C.\end{equation}  Of course, these integrals decay to zero as $t$ approaches infinity, but we will not use this fact.

Therefore, in view of (\ref{e_log_0}), (\ref{e_log_2}) and (\ref{e_log_3}), for each $t\geq \overline{t}$, for $C>0$ independent of $t$, \begin{equation}\label{e_log_4}H_{t,\omega}(0)\leq C\left(\log(t)+1\right).\end{equation}  In order to conclude, we observe that the bounds appearing in (\ref{e_green}), taking $T=\overline{t}$, imply that there exists $C>0$ independent of $\omega$ and $1\leq t\leq \overline{t}$ such that, for each $1\leq t\leq \overline{t}$, \begin{equation}\label{e_log_5}H_{t,\omega}(0)\leq C(\log(t)+1).\end{equation}  In view of (\ref{e_log_4}) and (\ref{e_log_5}), since $\omega\in\Omega$ was arbitrary, this completes the argument.  \end{proof}

The final proposition of this section demonstrates the primary use of Proposition \ref{e_log}, as will be seen in the Section 5 to follow.

\begin{prop}\label{e_inf}  Assume (\ref{steady}).  There exists $C>0$ satisfying $$\liminf_{n\rightarrow\infty} n\mathbb{E}_\pi\left(H_{n,\omega}(0)-H_{n-1,\omega}(0)\right)\leq C.$$\end{prop}

\begin{proof}  In view of Proposition \ref{e_log}, for each $n\geq 1$, for $C_2>0$ independent of $n$, \begin{equation}\label{e_inf_1}\mathbb{E}_\pi\left(H_{n,\omega}(0)\right)\leq C_2\left(\log(n)+1\right).\end{equation}  Choose $C_3>0$ such that, for each $n\geq 1$, $$C_2\left(\log(n)+1\right)+\mathbb{E}_\pi\left(H_{1,\omega}(0)\right)\leq C_3\sum_{m=1}^n\frac{1}{m}.$$  Using (\ref{e_inf_1}), we have, for each $n\geq 2$, $$\sum_{m=2}^n\mathbb{E}_\pi\left(H_{m,\omega}(0)-H_{m-1,\omega}(0)\right)+\mathbb{E}_\pi\left(H_{1,\omega}(0)\right)\leq C_3\sum_{m=1}^n\frac{1}{m},$$ and, therefore, for each $n\geq 2$, for $C>0$ independent of $n$, $$\sum_{m=2}^n\left(\mathbb{E}_\pi\left(H_{m,\omega}(0)-H_{m-1,\omega}(0)\right)-\frac{C_3}{m}\right)\leq -\mathbb{E}_\pi\left(H_{1,\omega}(0)\right)\leq C.$$  Proceeding by contradiction, this implies that, for infinitely many $m\geq 2$, $$\mathbb{E}_\pi\left(H_{m,\omega}(0)-H_{m-1,\omega}(0)\right)\leq \frac{C_3}{m},$$ which completes the argument.  \end{proof}  

\section{A Liouville Property for Strictly Sub-linear Solutions}

We now complete the proof of Theorem \ref{i_main} by proving that, on a subset of full probability, the constant functions are the only strictly sub-linear solutions $w:\mathbb{R}^d\rightarrow\mathbb{R}$ satisfying \begin{equation}\label{s_eq}-\frac{1}{2}\tr(A(x,\omega)D^2w)+b(x,\omega)\cdot Dw=0\;\;\textrm{on}\;\;\mathbb{R}^d,\end{equation} where we say that a strictly sub-linear $w:\mathbb{R}^d\rightarrow\mathbb{R}$ satisfies (\ref{s_eq}) if, for each $x\in\mathbb{R}^d$ and $t\geq 0$, \begin{equation}\label{s_sol} w(x)=\int_{\mathbb{R}^d}p_{t,\omega}(x,y)w(y)\;dy=P_{x,\omega}\left(w(X_t)\right).\end{equation}  The following argument is motivated by the analogous fact presented for the discrete setting in \cite{BDKY}.

\begin{thm}\label{s_main}  Assume (\ref{steady}) and (\ref{constants}).  On a subset of full probability, the constant functions are the the only strictly sub-linear $w:\mathbb{R}^d\rightarrow\mathbb{R}$ satisfying (\ref{s_sol}).\end{thm}

\begin{proof}  Fix $\omega\in\Omega$ and a strictly sub-linear $w:\mathbb{R}^d\rightarrow\mathbb{R}$ satisfying (\ref{s_sol}).  We begin with a preliminary computation.  For each $y\in\mathbb{R}^d$ and integer $n\geq 2$, $$\abs{w(0)-w(y)}=\left|\int_{\mathbb{R}^d}\left(p_{n,\omega}(0,z)-p_{n-1,\omega}(y,z)\right)w(z)\;dz\right|.$$  H\"older's inequality implies that, for each $y\in\mathbb{R}^d$, \begin{multline*}\abs{w(0)-w(y)}\leq  \\ \left(\int_{\mathbb{R}^d}\frac{\left(p_{n,\omega}(0,z)-p_{n-1,\omega}(y,z)\right)^2}{p_{n,\omega}(0,z)+p_{n-1,\omega}(y,z)}\;dz\right)^{1/2}\left(\int_{\mathbb{R}^d}w^2(z)(p_{n,\omega}(0,z)+p_{n-1,\omega}(y,z))\;dz\right)^{1/2}.\end{multline*}

In what follows, we use the fact that, after a Taylor expansion at $x=1$, for each $x>0$, $$2x\log(x)=2(x-1)+(x-1)^2-\int_1^x\frac{(x-s)^2}{s^2}\;ds\geq 2x-2+\frac{(x-1)^2}{x+1}.$$  This implies that, after suppressing the dependence on $z\in\mathbb{R}^d$, for each $y\in\mathbb{R}^d$ and integer $n\geq 2$, \begin{multline*}\left(\int_{\mathbb{R}^d}\frac{\left(p_{n,\omega}(0)-p_{n-1,\omega}(y)\right)^2}{p_{n,\omega}(0)+p_{n-1,\omega}(y)}\;dz\right)^{1/2}\leq \\ \left(\int_{\mathbb{R}^d}2p_{n-1,\omega}(y)\log(p_{n-1,\omega}(y))-2p_{n-1,\omega}(y)\log(p_{n,\omega}(0))-2p_{n-1,\omega}(y)+2p_{n,\omega}(0)\;dz\right)^{1/2}.\end{multline*}  And, therefore, after another application of H\"older's inequality, using the fact that, for each $z\in\mathbb{R}^d$, $$\int_{\mathbb{R}^d}p_{1,\omega}(0,y)p_{n-1,\omega}(y,z)\;dy=p_{n,\omega}(0,z),$$ we have, for each integer $n\geq 2$, \begin{multline*}\int_{\mathbb{R}^d}\abs{w(0)-w(y)}p_{1,\omega}(0,y)\;dy \leq \\ 2\left(H_{n,\omega}(0)-\int_{\mathbb{R}^d}p_{1,\omega}(0,y)H_{n-1,\omega}(y)\;dy\right)^{1/2}\left(\int_{\mathbb{R}^d}w^2(z)p_{n,\omega}(0,z)\;dz\right)^{1/2}.\end{multline*}  And, therefore, for each integer $n\geq 2$, using the fact that (\ref{stationary}) implies, for each $y\in\mathbb{R}^d$, $$H_{n-1,\omega}(y)=H_{n-1,\tau_y\omega}(0),$$ we have \begin{multline*}\int_{\mathbb{R}^d}\abs{w(0)-w(y)}p_{1,\omega}(0,y)\;dy \leq \\ 2\left(n\left(H_{n,\omega}(0)-P_{0,\omega}(H_{n-1,\tau_{X_1}\omega}(0))\right)\right)^{1/2}\left(\frac{1}{n}P_{0,\omega}\left(w^2(X_n)\right)\right)^{1/2}.\end{multline*}

Fix $\epsilon>0$.  The strict sub-linearity of $w$ implies that there exists $C=C(w,\epsilon)>0$ such that, for all $z\in\mathbb{R}^d$, $$w^2(z)\leq C+\epsilon\abs{z}^2.$$  Therefore, \begin{multline}\label{s_main_1}\int_{\mathbb{R}^d}\abs{w(0)-w(y)}p_{1,\omega}(0,y)\;dy \leq \\ \liminf_{n\rightarrow\infty} 2\left(n\left(H_{n,\omega}(0)-P_{0,\omega}(H_{n-1,\tau_{X_1}\omega}(0))\right)\right)^{1/2}\left(\frac{1}{n}P_{0,\omega}\left(\epsilon\abs{X_n}^2+C\right)\right)^{1/2}.\end{multline}

We observe that, in view of Proposition \ref{d_main}, for each $\omega\in\Omega_1$, for every $C>0$ and $\epsilon>0$, \begin{equation}\label{s_main_2}\lim_{n\rightarrow\infty}\left(\frac{1}{n}P_{0,\omega}\left(\epsilon\abs{X_n}^2+C\right)\right)^{1/2}=(\epsilon\overline{\alpha}d)^{1/2},\end{equation} and, using Fatou's lemma, (\ref{p_invariant_1}) and Proposition \ref{e_inf}, \begin{multline*}\mathbb{E}_\pi\left(\liminf_{n\rightarrow\infty}\left(n\left(H_{n,\omega}(0)-P_{0,\omega}(H_{n-1,\tau_{X_1}\omega}(0))\right)\right)\right)\leq \\ \liminf_{n\rightarrow\infty}\mathbb{E}_\pi\left(n\left(H_{n,\omega}(0)-P_{0,\omega}(H_{n-1,\tau_{X_1}\omega}(0))\right)\right)=\liminf_{n\rightarrow\infty}\mathbb{E}_\pi\left(n\left(H_{n,\omega}(0)-H_{n-1,\omega}(0)\right)\right)\leq C.\end{multline*}  This implies that there exists a subset $\Omega_2\subset\Omega$ of full probability such that, for each $\omega\in\Omega_2$ there exists $C(\omega)>0$ satisfying \begin{equation}\label{s_main_3}\liminf_{n\rightarrow\infty}\left(n\left(H_{n,\omega}(0)-P_{0,\omega}(H_{n-1,\tau_{X_1}\omega}(0))\right)\right)\leq C.\end{equation}

To conclude, we define the subset of full probability $$\Omega_3=\Omega_1\cap\Omega_2,$$ and observe that, whenever $\omega\in\Omega_3$ and $w:\mathbb{R}^d\rightarrow\mathbb{R}$ is strictly sub-linear and satisfies (\ref{s_sol}), we have $$\int_{\mathbb{R}^d}\abs{w(0)-w(y)}p_{1,\omega}(0,y)\;dy=0.$$  Since, for each $\omega\in\Omega$,  $p_{1,\omega}(0,y)\;dy$ is equivalent to Lebesgue measure, see \cite{Fr}, and since (\ref{s_sol}) implies that $w\in \C(\mathbb{R}^d)$, we have, for each $y\in\mathbb{R}^d$, $$w(y)=w(0).$$  This, since $\omega\in\Omega_3$ and $w$ satisfying (\ref{s_sol}) were arbitrary, completes the argument.  \end{proof}

\section{A Liouville Property for Bounded, Ancient Solutions}

We are now recall some aspects of \cite{F1} in order to prove Theorem \ref{i_ancient}, which states that, on a subset of full probability, the constant functions are the only bounded $w:\mathbb{R}^d\times(-\infty,\infty)\rightarrow\mathbb{R}$ satisfying \begin{equation}\label{a_eq} -\frac{1}{2}\tr(A(x,\omega)D^2w)+b(x,\omega)\cdot Dw=0\;\;\textrm{on}\;\;\mathbb{R}^d\times(-\infty,\infty).\end{equation}  Recall that we say a bounded $w:\mathbb{R}^d\times(-\infty,\infty)\rightarrow\mathbb{R}$ satisfies (\ref{a_eq}) if, for each $x\in\mathbb{R}^d$, $t\in(-\infty,\infty)$ and $s\geq 0$, \begin{equation}\label{a_sol} w(x,t+s)=\int_{\mathbb{R}^d}p_{s,\omega}(x,y)w(y,t)\;dy=P_{x,\omega}\left(w(X_s,t)\right).\end{equation}

In what follows, we will prove that, with high probability, the operators $R_{n+1}$ defined in (\ref{d_op}) are effectively averaged, or regularized, versions of the operators $R_n$.  Precisely, we show that, with probability approaching one as $n$ approaches infinity, the operator $$R_{n+1}\;\;\textrm{may be effectively compared with the operator}\;\;\overline{R}_n^{\ell_n^2-6}R_n^6.$$   We now define the subset on which we will obtain this comparison.

As before, it is necessary to obtain Controls \ref{Holder} and \ref{localization} on a large portion of space.  Define, for each $n\geq 0$, $$\tilde{F}_n=\left\{\;\omega\in\Omega\;|\;\omega\in B_n(x)\;\textrm{for all}\;x\in L_n\mathbb{Z}^d\cap[-2L_{n+2}^2, 2L_{n+2}^2]^d.\;\right\},$$ and, for each $n\geq 0$, \begin{equation}\label{o_subset} F_n=\tilde{F_n}\cap \tilde{F}_{n+1}\cap \tilde{F}_{n+2}.\end{equation}  The following proposition provides, for each $n\geq 0$, a lower bound for the probability of $F_n$.

\begin{prop}\label{a_probability}  Assume (\ref{steady}) and (\ref{constants}).  For each $n\geq 0$, for $C>0$ independent of $n$, $$\mathbb{P}(\Omega\setminus F_n)\leq CL_n^{(2(1+a)^2-1)d-M_0}.$$\end{prop}

\begin{proof}  In view of (\ref{mainevent1}), for each $n\geq 0$, for $C>0$ independent of $n$, $$\mathbb{P}(\Omega\setminus \tilde{F}_n)\leq \sum_{x\in L_n\mathbb{Z}^d\cap[-2L_{n+2}^2, 2L_{n+2}^2]^d}\mathbb{P}\left(\Omega\setminus B_n(x)\right)\leq C\left(L_{n+2}^2/L_n\right)^d\mathbb{P}\left(\Omega\setminus B_n(0)\right).$$  Therefore, using Theorem \ref{induction}, for each $n\geq 0$, for $C>0$ independent of $n$, $$\mathbb{P}(\Omega\setminus \tilde{F}_n)\leq CL_n^{(2(1+a)^2-1)d-M_0}.$$  This implies that, for each $n\geq 0$, \begin{multline*}\mathbb{P}\left(\Omega\setminus F_n\right)\leq \mathbb{P}\left(\Omega\setminus \tilde{F}_n\right)+\mathbb{P}\left(\Omega\setminus\tilde{F}_{n+1}\right)+\mathbb{P}\left(\Omega\setminus \tilde{F}_{n+2}\right) \\ \leq C\left(L_n^{(2(1+a)^2-1)d-M_0}+L_{n+1}^{(2(1+a)^2-1)d-M_0}+L_{n+2}^{(2(1+a)^2-1)d-M_0}\right)\leq CL_n^{(2(1+a)^2-1)d-M_0},\end{multline*} which completes the argument.  \end{proof}

We remark that, in view of (\ref{Holderexponent}) and (\ref{delta}), the exponent $(2(1+a)^2-1)d-M_0<0$.  Also, we remark that to simplify the notation in the definition of $\tilde{F}_n$, we obtain Controls \ref{Holder} and \ref{localization} on a somewhat larger portion of space than is strictly necessary for the argument to follow.

Before proceeding with the primary argument, we make an elementary observation concerning the regularizing properties of the kernels $\overline{R}_n$.  Notice the role of Theorem \ref{effectivediffusivity} in what follows.  Since we have a lower bound for the $\alpha_n$, we obtain an estimate uniform for $n\geq 0$.

\begin{prop}\label{a_regular}  Assume (\ref{steady}) and (\ref{constants}).  There exists $C>0$ satisfying, for each $n\geq 0$ and $f\in L^\infty(\mathbb{R}^d)$, $$\abs{\overline{R}_nf}_n\leq C\norm{f}_{L^\infty(\mathbb{R}^d)}.$$\end{prop}

\begin{proof}  Fix $n\geq 0$ and $f\in L^\infty(\mathbb{R}^d)$.  For each $x\in\mathbb{R}^d$, $$\overline{R}_nf(x)=\int_{\mathbb{R}^d}(4\pi \alpha_n L_n^2)^{-d/2}e^{-\abs{x-y}^2/4\alpha_n L_n^2}f(y)\;dy.$$  Therefore, \begin{equation}\label{o_regular_1}\norm{\overline{R}_nf(x)}_{L^\infty(\mathbb{R}^d)}\leq\norm{f}_{L^\infty(\mathbb{R}^d)}.\end{equation}  It remains to bound the H\"older semi-norm.

For each $x\in\mathbb{R}^d$, $$D\overline{R}_nf(x)=\pi^{-d/2}(4\alpha_n L_n^2)^{-1/2}\int_{\mathbb{R}^d}\frac{x-y}{(4\alpha_n L_n^2)^{(d+1)/2}}e^{-\abs{x-y}^2/4\alpha_n L_n^2}f(y)\;dy.$$  Therefore, in view of Theorem \ref{effectivediffusivity}, for each $x\in\mathbb{R}^d$, for $C>0$ independent of $n\geq 0$ and $f\in L^\infty(\mathbb{R}^d)$, $$\abs{D\overline{R}_nf(x)}=\left|\pi^{-d/2}(4\alpha_n L_n^2)^{-1/2}\int_{\mathbb{R}^d} ye^{-\abs{y}^2}f\left(4\alpha_nL_n^2)^{1/2}y+x\right)\;dy\right|\leq CL_n^{-1}\norm{f}_{L^\infty(\mathbb{R}^d)}.$$  So, whenever $x,y\in\mathbb{R}^d$ satisfy $0<\abs{x-y}<L_n$, \begin{equation}\label{o_regular_2} L_n^\beta\frac{\abs{\overline{R}_nf(x)-\overline{R}_nf(y)}}{\abs{x-y}^\beta}\leq CL_n^{\beta-1}\norm{f}_{L^\infty(\mathbb{R}^d)}\abs{x-y}^{1-\beta}\leq \norm{f}_{L^\infty(\mathbb{R}^d)}.\end{equation}  And, in view of (\ref{o_regular_1}), if $\abs{x-y}\geq L_n$, \begin{equation}\label{o_regular_3}  L_n^\beta\frac{\abs{\overline{R}_nf(x)-\overline{R}_nf(y)}}{\abs{x-y}^\beta}\leq 2\norm{f}_{L^\infty(\mathbb{R}^d)}.\end{equation}  The claim follows from (\ref{o_regular_1}), (\ref{o_regular_2}) and (\ref{o_regular_3}).  \end{proof}

The following proposition, which is a somewhat simplified version of Proposition 3.9 appearing in \cite{F1},  provides, on a large portion of space, on the subset $F_n$ defined in (\ref{o_subset}), an effective comparison between the operators $$R_{n+1}\;\;\textrm{and}\;\;\overline{R}_n^{\ell_n^2-6}R_n^6.$$  Notice that the estimates contained in the following proposition depend on the unscaled, $\beta$-H\"older norm of the initial data.  This is not an issue since it is well-known, and shown in Proposition \ref{a_weak} below, that any bounded function $w:\mathbb{R}^d\times(-\infty,\infty)\rightarrow\mathbb{R}$ satisfying (\ref{a_sol}) is H\"older continuous in space, uniformly in time.

\begin{prop}\label{a_main}  Assume (\ref{steady}) and (\ref{constants}).  For each $n\geq 0$, $\omega\in F_n$ and $f\in C^{0,\beta}(\mathbb{R}^d)$, for $C>0$ independent of $n$, $$\sup_{x\in B_{\tilde{D}_{n+1}}}\abs{R_{n+1}f(x,\omega)-\left(\overline{R}_n\right)^{\ell_n^2-6}\left(R_n\right)^6f(x,\omega)}\leq CL_n^{\beta-7(\delta-2a)}\norm{f}_{C^{0,\beta}(\mathbb{R}^d)}.$$\end{prop}

\begin{proof}  Fix $n\geq 0$, $\omega\in F_n$ and $f\in C^{0,\beta}(\mathbb{R}^d)$.  In what follows, we suppress the dependence on $\omega\in\Omega$.  Notice that (\ref{D_1}) implies that, in the definition of $F_n$, we have \begin{equation}\label{o_main_0} 3\tilde{D}_{n+2}< L_{n+2}^2.\end{equation}  And, for what follows, we recall that $$\left(R_{n+1}\right)f(x)=\left(R_n\right)^{\ell_n^2}f(x).$$

Fix $x\in B_{\tilde{D}_{n+1}}$ and define the cutoff function $\tilde{\chi}_{n,x}:\mathbb{R}^d\rightarrow\mathbb{R}$, recalling (\ref{cutoff}), \begin{equation}\label{o_main_cut}\tilde{\chi}_{n,x}(y)=\chi_{2\tilde{D}_{n+2}}(y-x)\;\;\textrm{on}\;\;\mathbb{R}^d.\end{equation}  Since \begin{equation}\label{o_main_1} \norm{R_nf}_{L^\infty(\mathbb{R}^d)}\leq\norm{f}_{L^\infty(\mathbb{R}^d)},\end{equation} and since $x\in B_{\tilde{D}_{n+1}}$ and $\omega\in F_n$, Control \ref{localization}, Proposition \ref{local11} and (\ref{o_main_0}) imply that $$\abs{\left(R_n\right)^{\ell_n^2}f(x)-\left(R_n\right)^{\ell_n^2-1}\tilde{\chi}_{n,x}R_nf(x)}=\abs{\left(R_n\right)^{\ell_n^2-1}(1-\tilde{\chi}_{n,x})R_nf(x)}\leq e^{-\kappa_{n+2}}\norm{f}_{L^\infty(\mathbb{R}^d)}.$$  Proceeding inductively, we conclude that \begin{equation}\label{o_main_2} \abs{\left(R_n\right)^{\ell_n^2}f(x)-\left(\tilde{\chi}_{n,x}R_n\right)^{\ell_n^2}f(x)}\leq \ell_n^2e^{-\kappa_{n+2}}\norm{f}_{L^\infty(\mathbb{R}^d)}.\end{equation}

We now write $$\left(\tilde{\chi}_{n,x}R_n\right)^{\ell_n^2}f(x)=\left(\tilde{\chi}_{n,x}S_n+\tilde{\chi}_{n,x}\overline{R}_n\right)^{\ell_n^2}f(x),$$ and, for nonnegative integers $k_i\geq 0$, \begin{multline*}\left(\tilde{\chi}_{n,x}S_n+\tilde{\chi}_{n,x}\overline{R}_n\right)^{\ell_n^2}f(x)= \\ \sum_{m=0}^{\ell_n^2}\sum_{k_0+\ldots+k_m+m=\ell_n^2}\left(\tilde{\chi}_{n,x}\overline{R}_n\right)^{k_0}\tilde{\chi}_{n,x}S_n\left(\tilde{\chi}_{n,x}\overline{R}_n\right)^{k_1} \ldots \tilde{\chi}_{n,x}S_n\left(\tilde{\chi}_{n,x}\overline{R}_n\right)^{k_m}f(x).\end{multline*}

Since, for each $n\geq 0$, $$\abs{f}_n\leq L_n^\beta\norm{f}_{C^{0,\beta}(\mathbb{R}^d)},$$ and since $x\in B_{\tilde{D}_{n+1}}$ and $\omega\in F_n$, Control \ref{Holder}, Proposition \ref{prelim_product}, Proposition \ref{prelim_extension}, Proposition \ref{d_contract}, Proposition \ref{a_regular} and (\ref{o_main_0}) imply that \begin{multline*}\left|\sum_{m=7}^{\ell_n^2}\sum_{k_0+\ldots+k_m+m=\ell_n^2}\left(\tilde{\chi}_{n,x}\overline{R}_n\right)^{k_0}\tilde{\chi}_{n,x}S_n\left(\tilde{\chi}_{n,x}\overline{R}_n\right)^{k_1} \ldots \tilde{\chi}_{n,x}S_n\left(\tilde{\chi}_{n,x}\overline{R}_n\right)^{k_m}f(x)\right|\leq \\ \sum_{m=7}^{\ell_n^2} {\ell_n^2 \choose m} 3^mL_n^{\beta-m\delta}\norm{f}_{C^{0,\beta}(\mathbb{R}^d)}.  \end{multline*}  Therefore, for $C>0$ independent of $n$, the lefthand side of the above string of inequalities is bounded by \begin{equation}\label{o_main_3} \sum_{m=7}^{\ell_n^2} \frac{3^m}{m!}L_n^{\beta-m(\delta-2a)}\norm{f}_{C^{0,\beta}(\mathbb{R}^d)}\leq CL_n^{\beta-7(\delta-2a)}\norm{f}_{C^{0,\beta}(\mathbb{R}^d)},\end{equation} where we remark that $\beta-7(\delta-2a)<0$ in view of (\ref{Holderexponent}) and (\ref{delta}).

It remains to consider \begin{equation}\label{o_main_4}\sum_{m=0}^{6}\sum_{k_0+\ldots+k_m+m=\ell_n^2}\left(\tilde{\chi}_{n,x}\overline{R}_n\right)^{k_0}\tilde{\chi}_{n,x}S_n\left(\tilde{\chi}_{n,x}\overline{R}_n\right)^{k_1} \ldots \tilde{\chi}_{n,x}S_n\left(\tilde{\chi}_{n,x}\overline{R}_n\right)^{k_m}f(x).\end{equation}  We will prove that, up to an error which vanishes as $n$ approaches infinity, the above sum reduces to $$\left(\overline{R}_n\right)^{\ell_n^2-6}\left(R_n\right)^6f(x).$$  To do so, we consider each summand in $m$ individually.

For the case $m=0$, the single summand is \begin{equation}\label{o_main_5} \left(\tilde{\chi}_{n,x}\overline{R}_n\right)^{\ell_n^2}f(x).\end{equation}

For the case $m=1$, observe that, since $x\in B_{\tilde{D}_{n+1}}$ and $\omega\in F_n$, Control \ref{Holder}, Proposition \ref{prelim_product}, Proposition \ref{prelim_extension}, Proposition \ref{d_contract}, Proposition \ref{a_regular} and (\ref{o_main_0}) imply that, for $C>0$ independent of $n$,  \begin{multline}\label{o_main_7}\left|\sum_{k_0+k_1+1=\ell_n^2}\left(\tilde{\chi}_{n,x}\overline{R}_n\right)^{k_0}\tilde{\chi}_{n,x}S_n\left(\tilde{\chi}_{n,x}\overline{R}_n\right)^{k_1}f(x)-\left(\tilde{\chi}_{n,x}\overline{R}_n\right)^{\ell_n^2-1}\tilde{\chi}_{n,x}S_nf(x)\right| \\ \leq C{\ell_n^2 \choose 1}3L_n^{-\delta}\norm{f}_{L^\infty(\mathbb{R}^d)}\leq CL_n^{2a-\delta}\norm{f}_{L^\infty(\mathbb{R}^d)},\end{multline} where we observe that $2a-\delta<0$ in view of (\ref{Holderexponent}) and (\ref{delta}).  Furthermore, \begin{equation}\label{o_main_8} \left(\tilde{\chi}_{n,x}\overline{R}_n\right)^{\ell_n^2-1}\tilde{\chi}_{n,x}S_nf(x)=\left(\tilde{\chi}_{n,x}\overline{R}_n\right)^{\ell_n^2-1}\tilde{\chi}_{n,x}R_nf(x)-\left(\tilde{\chi}_{n,x}\overline{R}_n\right)^{\ell_n^2}f(x).\end{equation}  Notice the cancellation between (\ref{o_main_5}) and (\ref{o_main_8}).

In what follows, we use that fact that, for every $f\in L^\infty(\mathbb{R}^d)$, $$\norm{S_nf}_{L^\infty(\mathbb{R}^d)}\leq 2\norm{f}_{L^\infty(\mathbb{R}^d)}.$$  Fix $2\leq m\leq 6$.  In this case, as in the case $m=0$ and $m=1$, Control \ref{Holder}, Proposition \ref{prelim_product}, Proposition \ref{prelim_extension}, Proposition \ref{d_contract} and Proposition \ref{a_regular} allow us to reduce the sum to the single term $k_i=0$ for all $1\leq i\leq m$.  Observe that, since $x\in B_{\tilde{D}_{n+1}}$ and $\omega\in F_n$, for $C>0$ independent of $n$, $$\left|\sum_{k_m\neq 0}\left(\tilde{\chi}_{n,x}\overline{R}_n\right)^{k_0}\tilde{\chi}_{n,x}S_n\ldots\left(\tilde{\chi}_{n,x}\overline{R}_n\right)^{k_m}f(x)\right| \leq C{\ell_n^2 \choose m}3^mL_n^{-m\delta}\norm{f}_{L^\infty(\mathbb{R}^d)}.$$  And, generally, for $1\leq i\leq m$, since $x\in B_{\tilde{D}_{n+1}}$ and $\omega\in F_n$, for $C>0$ independent of $n$,\begin{multline*}\left|\sum_{k_i\neq 0,\;k_j=0\;\textrm{if}\;j>i}\left(\tilde{\chi}_{n,x}\overline{R}_n\right)^{k_0}\tilde{\chi}_{n,x}S_n\ldots\left(\tilde{\chi}_{n,x}\overline{R}_n\right)^{k_i}\left(\tilde{\chi}_{n,x}S_n\right)^{m-i}f(x)\right| \\ \leq C{\ell_n^2-m+i \choose i}2^{m-i}3^iL_n^{-i\delta}\norm{f}_{L^\infty(\mathbb{R}^d)}.\end{multline*}  Therefore, for $C>0$ independent of $2\leq m\leq 6$ and $n$, \begin{multline}\label{o_main_9}\left|\sum_{k_0+\ldots+k_m+m=\ell_n^2}\left(\tilde{\chi}_{n,x}\overline{R}_n\right)^{k_0}\ldots\left(\tilde{\chi}_{n,x}\overline{R}_n\right)^{k_m}f(x)-\left(\tilde{\chi}_{n,x}\overline{R}_n\right)^{\ell_n^2-m}\left(\tilde{\chi}_{n,x}S_n\right)^mf(x)\right| \\ \leq C\sum_{i=1}^m{\ell_n^2-m+i \choose i}2^{m-i}3^iL_n^{-i\delta}\norm{f}_{L^\infty(\mathbb{R}^d)}\leq CL_n^{2a-\delta}\norm{f}_{L^\infty(\mathbb{R}^d)},\end{multline} where we observe that $2a-\delta<0$ in view of (\ref{Holderexponent}) and (\ref{delta}).

Furthermore, again using Control \ref{Holder}, Proposition \ref{prelim_product}, Proposition \ref{prelim_extension}, Proposition \ref{d_contract}, Proposition \ref{a_regular} and (\ref{o_regular_1}), since $x\in B_{\tilde{D}_{n+1}}$ and $\omega\in F_n$, for each $2\leq m\leq 6$, for $C>0$ independent of $n$, \begin{multline*}\left|\left(\tilde{\chi}_{n,x}\overline{R}_n\right)^{\ell_n^2-m}\left(\tilde{\chi}_{n,x}S_n\right)^mf(x)-\left(\tilde{\chi}_{n,x}\overline{R}_n\right)^{\ell_n^2-m}\left(\tilde{\chi}_{n,x}S_n\right)^{m-1}\tilde{\chi}_{n,x}R_nf(x)\right| \\ =\left|\left(\tilde{\chi}_{n,x}\overline{R}_n\right)^{\ell_n^2-m}\left(\tilde{\chi}_{n,x}S_n\right)^{m-1}\tilde{\chi}_{n,x}\overline{R}_nf(x)\right|\leq C 3^{m-1}L_n^{(m-1)\delta}\norm{f}_{L^\infty(\mathbb{R}^d)}.\end{multline*}  Proceeding inductively, for each $2\leq m\leq 6$, for $C>0$ independent of $m$ and $n$, $$\left|\left(\tilde{\chi}_{n,x}\overline{R}_n\right)^{\ell_n^2-m}\left(\tilde{\chi}_{n,x}S_n\right)^mf(x)-\left(\tilde{\chi}_{n,x}\overline{R}_n\right)^{\ell_n^2-m}\tilde{\chi}_{n,x}S_n\left(\tilde{\chi}_{n,x}R_n\right)^{m-1}f(x)\right|\leq CL_n^{-\delta}\norm{f}_{L^\infty(\mathbb{R}^d)},$$ where we observe that \begin{multline}\label{o_main_10}\left(\tilde{\chi}_{n,x}\overline{R}_n\right)^{\ell_n^2}f(x)+\sum_{m=1}^6\left(\tilde{\chi}_{n,x}\overline{R}_n\right)^{\ell_n^2-m}\tilde{\chi}_{n,x}S_n\left(\tilde{\chi}_{n,x}R_n\right)^{m-1}f(x) \\ =\left(\tilde{\chi}_{n,x}\overline{R}_n\right)^{\ell_n^2-6}\left(\tilde{\chi}_{n,x}R_n\right)^{6}f(x).\end{multline}  And, since $x\in B_{\tilde{D}_{n+1}}$ and $\omega\in F_n$, Control \ref{localization}, Proposition \ref{local11}, Proposition \ref{d_hlocalize} and (\ref{o_main_0}) imply that there exists $C>0$ and $c>0$ independent of $n$ and such that \begin{equation}\label{o_main_11}\left|\left(\tilde{\chi}_{n,x}\overline{R}_n\right)^{\ell_n^2-6}\left(\tilde{\chi}_{n,x}R_n\right)^6f(x)-\left(\overline{R}_n\right)^{\ell_n^2-6}\left(R_n\right)^6f(x)\right|\leq C\ell_n^2e^{-c\kappa_{n+2}}\norm{f}_{L^\infty(\mathbb{R}^d)}.\end{equation}

Therefore, in view of (\ref{o_main_0}), (\ref{o_main_2}), (\ref{o_main_3}), (\ref{o_main_5}), (\ref{o_main_7}), (\ref{o_main_9}), (\ref{o_main_10}) and (\ref{o_main_11}), there exits $C>0$ and $c>0$ independent of $n$ such that \begin{multline}\label{o_main_12} \abs{R_{n+1}f(x)-\left(\overline{R}_n\right)^{\ell_n^2-6}\left(R_n\right)^6f(x)}=\abs{\left(R_n\right)^{\ell_n^2}f(x)-\left(\overline{R}_n\right)^{\ell_n^2-6}\left(R_n\right)^6f(x)} \\ \leq C\ell_n^2e^{-c\kappa_{n+2}}\norm{f}_{L^\infty(\mathbb{R}^d)}+CL_n^{\beta-7(\delta-2a)}\norm{f}_{C^{0,\beta}(\mathbb{R}^d)}+CL_n^{2a-\delta}\norm{f}_{L^\infty(\mathbb{R}^d)}.\end{multline}  In view of (\ref{Holderexponent}), (\ref{L}), (\ref{kappa}) and (\ref{delta}) there exits $C>0$ independent of $n$ such that, for all $n\geq 0$, $$\ell_n^2e^{-c\kappa_{n+2}}\leq CL_n^{\beta-7(\delta-2a)}\;\;\textrm{and}\;\;L_n^{2a-\delta}\leq CL_n^{\beta-7(\delta-2a)}.$$  And, since $\norm{f}_{L^\infty(\mathbb{R}^d)}\leq \norm{f}_{C^{0,\beta}(\mathbb{R}^d)}$, we have, using (\ref{o_main_12}), for $C>0$ independent of $n$, \begin{equation}\label{o_main_14}\abs{R_{n+1}f(x,\omega)-\left(\overline{R}_n\right)^{\ell_n^2-6}\left(R_n\right)^6f(x,\omega)}\leq CL_n^{\beta-7(\delta-2a)}\norm{f}_{C^{0,\beta}(\mathbb{R}^d)}.\end{equation}  Since $n\geq 0$, $\omega\in F_n$, $x\in B_{\tilde{D}_{n+1}}$ and $f\in C^{0,\beta}(\mathbb{R}^d)$ were arbitrary, this completes the proof.  \end{proof}

Because the bound appearing on Proposition \ref{a_main} depends on the unscaled $\beta$-H\"older norm of the initial data, we now observe that any bounded function $w:\mathbb{R}^d\times(-\infty,\infty)\rightarrow\mathbb{R}$ satisfying (\ref{a_sol}) is $\beta$-H\"older continuous in space, uniformly in time.

\begin{prop}\label{a_weak}  Assume (\ref{steady}).  For each $\omega\in\Omega$, $t\geq 1$ and $g\in L^\infty(\mathbb{R}^d)$, for $C>0$ independent of $\omega\in\Omega$ and $t\geq 1$, $$\norm{R_tg(x,\omega)}_{C^{0,\beta}(\mathbb{R}^d)}\leq C\norm{g}_{L^\infty(\mathbb{R}^d)}.$$\end{prop}

\begin{proof}  Fix $\omega\in\Omega$ and $g\in L^\infty(\mathbb{R}^d)$.  Recall that, for each $t\geq 0$ and $x\in\mathbb{R}^d$, see \cite{Fr}, \begin{equation}\label{o_weak_1}R_tg(x,\omega)=P_{x,\omega}\left(g(X_t)\right)=\int_{\mathbb{R}^d}p_{t,\omega}(x,y)g(y)\;dy,\end{equation} for $p_{t,\omega}(x,y):\mathbb{R}^d\times\mathbb{R}^d\times(0,\infty)\rightarrow\mathbb{R}$ satisfying, for each $0<t\leq 1$, for $C>0$ and $c>0$ independent of $\omega$,\begin{equation}\label{o_weak_2} \abs{p_{t,\omega}(x,y)}\leq Ct^{-d/2}e^{-c\abs{x-y}^2/t}\;\;\textrm{and}\;\;\abs{D_xp_{t,\omega}(x,y)}\leq Ct^{-(d+1)/2}e^{-c\abs{x-y}^2/t}.\end{equation}

First, we observe that for each $x\in\mathbb{R}^d$ and $t\geq 0$, using (\ref{o_weak_1}), \begin{equation}\label{o_weak_3} \abs{R_tg(x,\omega)}\leq \norm{g}_{L^\infty(\mathbb{R}^d)}.\end{equation}   It remains to bound the H\"older semi-norm.

Whenever $x,y\in\mathbb{R}^d$ satisfy $\abs{x-y}\geq 1$, \begin{equation}\label{o_weak_4} \abs{R_1g(x,\omega)-R_1g(y,\omega)}\leq 2\norm{g}_{L^\infty(\mathbb{R}^d)}\leq 2\abs{x-y}^\beta\norm{g}_{L^\infty(\mathbb{R}^d)}.\end{equation}  And, whenever $x,y\in\mathbb{R}^d$ satisfy $\abs{x-y}<1$, in view of (\ref{o_weak_1}) and (\ref{o_weak_2}), for $C>0$ independent of $\omega\in\Omega$, \begin{equation}\label{o_weak_5}\abs{R_1g(x,\omega)-R_1g(y,\omega)}\leq C\abs{x-y}\norm{g}_{L^\infty(\mathbb{R}^d)}\leq C\abs{x-y}^\beta\norm{g}_{L^\infty(\mathbb{R}^d)}.\end{equation}  Therefore, for each $x,y\in\mathbb{R}^d$ and $t\geq 1$, using (\ref{o_weak_3}), (\ref{o_weak_4}) and (\ref{o_weak_5}), \begin{multline}\label{o_weak_6} \abs{R_tg(x,\omega)-R_tg(y,\omega)}=\abs{R_{t-1}(R_1g(x,\omega)-R_1g(y,\omega))} \\ \leq \sup_{x,y\in\mathbb{R}^d}\abs{R_1g(x,\omega)-R_1g(y,\omega)}\leq C\abs{x-y}^\beta\norm{g}_{L^\infty(\mathbb{R}^d)}.\end{multline} The claim follows from (\ref{o_weak_3}), (\ref{o_weak_4}) and (\ref{o_weak_6}), since $\omega\in\Omega$ and $g\in L^\infty(\mathbb{R}^d)$ were arbitrary.  \end{proof}

We now use Proposition \ref{a_main} to prove Theorem \ref{i_ancient}. To do so, we recall the subsets $\left\{F_n\right\}_{n=0}^\infty$ defined in (\ref{o_subset}), and observe in view of (\ref{Holderexponent}), (\ref{delta}) and Proposition \ref{a_probability}, for $C>0$ independent of $n$, $$\sum_{n=0}^\infty\mathbb{P}\left(\Omega\setminus F_n\right)\leq \sum_{n=0}^\infty CL_n^{(2(1+a)^2-1)d-M_0}<\infty.$$  Therefore, using the Borel-Cantelli lemma, we define the subset $\Omega_2\subset\Omega$ of full probability \begin{equation}\label{lp_subset} \Omega_2=\left\{\;\omega\in\Omega\;|\;\textrm{There exists}\;\overline{n}(\omega)\geq 0\;\textrm{such that}\;\omega\in F_n\;\textrm{for all}\;n\geq\overline{n}.\;\right\},\end{equation} and prove that, on this subset, the only bounded functions satisfying (\ref{a_sol}) are the constant functions.

\begin{thm}\label{lp_main}  Assume (\ref{steady}) and (\ref{constants}).  For each $\omega\in\Omega_2$, the constant functions are the only bounded $w:\mathbb{R}^d\times(-\infty,\infty)\rightarrow\mathbb{R}$ satisfying (\ref{a_sol}).\end{thm}

\begin{proof}  Fix $\omega\in\Omega_2$ and a bounded function $w:\mathbb{R}^d\times(-\infty,\infty)\rightarrow\mathbb{R}$ satisfying (\ref{a_sol}).  We write, for each $t\in(-\infty,\infty)$, \begin{equation}\label{lp_main_11}w_t(x)=w(x,t),\end{equation} and observe that (\ref{a_sol}) implies, for each $t\in (-\infty,\infty)$ and $x\in\mathbb{R}^d$, \begin{equation}\label{lp_main_00} R_1w_{t-1}(x,\omega)=w_t(x).\end{equation}  Therefore, in view of Proposition \ref{a_weak}, for each $t\in(-\infty,\infty)$, for $C>0$ independent of $t$, \begin{equation}\label{lp_main_0}\norm{w_t}_{C^{0,\beta}(\mathbb{R}^d)}\leq C\norm{w}_{L^\infty(\mathbb{R}^d\times(-\infty,\infty))}.\end{equation}

Fix $t\in (-\infty,\infty)$.  Since $\omega\in\Omega_2$, fix $\overline{n}\geq 0$ such that $\omega\in A_n$ for every $n\geq \overline{n}$.  Because, for each $n\geq 0$ and $k\geq 0$, $$\left(R_n\right)^kw_{t-kL_n^2}(x,\omega)=w_t(x)\;\;\textrm{on}\;\;\mathbb{R}^d,$$ for every $n\geq \overline{n}$, Proposition \ref{a_main}, (\ref{lp_main_00}) and (\ref{lp_main_0}) imply that, for $C>0$ independent of $n\geq\overline{n}$,   \begin{multline}\label{lp_main_2}  \sup_{x\in B_{\tilde{D}_{n+1}}}\abs{R_{n+1}w_{t-L_{n+1}^2}(x)-\left(\overline{R}_n\right)^{\ell_n^2-6}\left(R_n\right)^6w_{t-L_{n+1}^2}(x,\omega)}\leq CL_n^{\beta-7(\delta-2a)}\norm{w_{t-L_{n+1}^2}}_{C^{0,\beta}(\mathbb{R}^d)} \\ \leq CL_n^{\beta-7(\delta-2a)}\norm{w}_{L^\infty(\mathbb{R}^d\times(-\infty,\infty))}.\end{multline}  And, since $w$ satisfies (\ref{a_sol}), for $C>0$ independent of $n\geq\overline{n}$, we have \begin{equation}\label{lp_main_4} \sup_{x\in B_{\tilde{D}_{n+1}}}\abs{w_t(x)-\left(\overline{R}_n\right)^{\ell_n^2-6}w_{t+6L_n^2-L_{n+1}^2}(x)}\leq CL_n^{\beta-7(\delta-2a)}\norm{w}_{L^\infty(\mathbb{R}^d\times(-\infty,\infty))}.\end{equation}

Using the same computation that appears in Proposition \ref{a_regular}, using Theorem \ref{effectivediffusivity}, for each $n\geq 0$, for $C>0$ independent of $n$, \begin{equation}\label{lp_main_5} \norm{D_x(\overline{R}_n)^{\ell_n^2-6}w_{t+6L_n^2-L_{n+1}^2}}_{L^\infty(\mathbb{R}^d)}\leq CL_{n+1}^{-1}\norm{w}_{L^\infty(\mathbb{R}^d\times(-\infty,\infty))}.\end{equation}  Furthermore, since for each $n\geq 0$, \begin{equation}\label{lp_main_6} \norm{(\overline{R}_n)^{\ell_n^2-6}w_{t+6L_n^2-L_{n+1}^2}}_{L^\infty(\mathbb{R}^d)}\leq \norm{w}_{L^\infty(\mathbb{R}^d\times(-\infty,\infty))}, \end{equation} in combination (\ref{lp_main_5}) and (\ref{lp_main_6}) imply that, along subsequence $\left\{n_k\rightarrow\infty\right\}_{k=1}^\infty$, for a constant $\overline{w}_t\in\mathbb{R}$, as $k\rightarrow\infty$, \begin{equation}\label{lp_main_7} (\overline{R}_{n_k})^{\ell_{n_k}^2-6}w_{t+6L_{n_k}^2-L_{n_k+1}^2}\rightarrow\overline{w}_t\;\;\textrm{locally uniformly on}\;\;\mathbb{R}^d.\end{equation}  Therefore, in view of (\ref{lp_main_4}) and (\ref{lp_main_7}), $$w_t=\overline{w}_t\;\;\textrm{on}\;\;\mathbb{R}^d.$$

Since $t\in(-\infty,\infty)$ was arbitrary, we conclude that, for each $t\in(-\infty,\infty)$, there exists $\overline{w}_t\in\mathbb{R}$ satisfying $$w_t=\overline{w}_t\;\;\textrm{on}\;\;\mathbb{R}^d.$$  And, whenever $t_1<t_2$, (\ref{a_sol}) implies that, for each $x\in\mathbb{R}^d$, $$\overline{w}_{t_1}=R_{t_2-t_1}w_{t_1}(x,\omega)=w_{t_2}(x)=\overline{w}_{t_2}.$$ We therefore conclude that $w:\mathbb{R}^d\times(-\infty,\infty)\rightarrow\mathbb{R}$ is constant, which, since $\omega\in\Omega_2$ and $w$ satisfying (\ref{a_sol}) were arbitrary, completes the argument.  \end{proof}

\bibliography{Liouville}
\bibliographystyle{plain}

\end{document}